\documentclass[11pt]{article}
\usepackage[utf8]{inputenc}
\usepackage[T1]{fontenc}
\usepackage{fixltx2e}
\usepackage{graphicx}
\usepackage{grffile}
\usepackage{longtable}
\usepackage{wrapfig}
\usepackage{rotating}
\usepackage[normalem]{ulem}
\usepackage{amsmath}
\usepackage{textcomp}
\usepackage{amssymb}
\usepackage{capt-of}
\usepackage{hyperref}
\usepackage{tikz, tikz-cd}
\usepackage[all,arc,curve,frame,color]{xy}
\usepackage{placeins}
\usepackage{amsthm}
\usepackage{titling}
\usepackage{standalone}
\newtheorem{theorem}{Theorem}
\newtheorem*{thm}{Theorem}
\newtheorem{prop}{Proposition}
\newtheorem{conjecture}{Conjecture}
\newtheorem*{conj}{Conjecture}

\theoremstyle{remark}
\newtheorem{remark}{Remark}
\usepackage{fullpage,enumerate,amssymb}
\usepackage[alphabetic,backrefs]{amsrefs} 
\newcommand{\C}{{\mathbb C}}
\newcommand{\F}{{\mathbb F}}
\newcommand{\PP}{{\mathbb P}}
\newcommand{\Q}{{\mathbb Q}}
\newcommand{\Z}{{\mathbb Z}}
\newcommand{\calD}{{\mathcal D}}
\author{Solomon Vishkautsan \\
with an appendix by Michael Stoll}
\date{}
\title{Quadratic rational functions with a rational periodic critical point of period \(3\)}
\hypersetup{
 pdfauthor={Solomon Vishkautsan \\
with an appendix by Michael Stoll},
 pdftitle={Quadratic rational functions with a rational periodic critical point of period \(3\)},
 pdfkeywords={},
 pdfsubject={},
 pdfcreator={Emacs 25.1.1 (Org mode 8.3.4)}, 
 pdflang={English}}
\begin{document}

\maketitle
\begin{abstract}
We provide a complete classification of possible graphs of rational preperiodic points of quadratic rational functions defined over the rationals with a rational periodic critical point of period 3, under two assumptions: that these functions have no periodic points of period at least 5 and the conjectured enumeration of points on a certain genus 6 affine plane curve. We show that there are exactly six such possible graphs, and that rational functions satisfying the conditions above have at most eleven rational preperiodic points.
\end{abstract}

\section{Introduction}
\label{sec:orgheadline1}

In this article we continue the classification of preperiodicity graphs of quadratic rational functions defined over \(\mathbb{Q}\) with a \(\mathbb{Q}\)-rational periodic critical point that was begun in \cite{MR3656200}. Our aim in this article is to provide a complete classification of rational quadratic functions defined over \(\mathbb{Q}\) with a \(\mathbb{Q}\)-rational periodic critical point of period \(3\).

Let \(\phi:\mathbb{P}^1\to\mathbb{P}^1\) be a rational function, defined over some base field \(K\). A point \(P\in\mathbb{P}^1\) is called \emph{periodic} for \(\phi\) if there exists a positive integer \(n\) such that \(\phi^n(P)=P\). The minimal such \(n\) is called the \emph{period} of \(P\). A point \(P\in\mathbb{P}^1\) is called \emph{preperiodic} for \(\phi\) if there exists a nonnegative integer \(m\) such that \(\phi^m(P)\) is periodic for \(\phi\).

One of the main motivations for this article is the following conjecture of Morton and Silverman \cite{MR1264933}.

\begin{conj}[Uniform Boundedness Conjecture] Let $\phi:\mathbb{P}^N\to\mathbb{P}^N$ be a morphism of degree $d\geq{2}$ defined over a number field $K$. Then 
the number of $K$-rational preperiodic points of $\phi$ is bounded by a bound depending only on $N$ and the degrees of $K/\mathbb{Q}$ and $\phi$.
\end{conj}

This powerful conjecture implies in particular Merel's theorem of uniform boundedness of torsion subgroups of the Mordell--Weil groups of rational points on elliptic curves (see \cite{MR2316407}, Remark 3.19), as well as the conjectured uniform boundedness of torsion subgroups of abelian varieties (see \cite{MR1995861}).

By Northcott's theorem \cite{MR0034607}, the set \(\text{PrePer}(\phi, K)\) of \(K\)-rational preperiodic points of a morphism \(\phi:\mathbb{P}^N\to\mathbb{P}^N\) defined over a number field \(K\), is finite. It can therefore be given a finite directed graph structure, called the \emph{preperiodicity graph} of \(\phi\), by drawing an arrow from \(P\) to \(\phi(P)\) for each \(P\in \text{PrePer}(\phi, K)\). The conjecture of Morton and Silverman implies in particular that the number of possible preperiodicity graphs for such \(\phi\) is finite and depends only on the constants \(\deg(\phi), [K:\mathbb{Q}]\) and \(N\). It is therefore natural to ask for the list of all possible preperiodicity graphs for a given family of endomorphisms of \(\mathbb{P}^N\).

The main example of a classification of preperiodicity graphs for rational functions is the classification of preperiodicity graphs of quadratic polynomials over \(\mathbb{Q}\) by Poonen \cite{MR1617987}. It relies on the following conjecture.

\begin{conj}[Flynn, Poonen and Schaefer \cite{MR1480542}] Let $\phi$ be a quadratic polynomial defined over $\mathbb{Q}$. Then $\phi$ has no $\mathbb{Q}$-rational periodic points of period greater than \(3\).
\end{conj}

Evidence for this conjecture can be found in several articles, including \cite{MR1480542,MR1665198,MR2465796,MR3065461}.

\begin{thm}[Poonen]
Assuming the Flynn, Poonen and Schaefer conjecture, there exist exactly $12$ possible preperiodicity graphs for quadratic polynomials defined over $\mathbb{Q}$, and a quadratic polynomial has at most nine $\mathbb{Q}$-rational preperiodic points.
\end{thm}

In the article \cite{MR3656200}, J.K. Canci and this author proposed a generalization of Poonen's classification of preperiodicity graphs of quadratic polynomials to the classification of preperiodicity graphs of quadratic rational functions with a rational periodic critical point. Recall that a point \(P\in\mathbb{P}^1\) is called \emph{critical} if \(\phi'(P) = 0\) (if \(P\) or \(\phi\)(P) are \(\infty\) then we conjugate \(\phi\) by an automorphism of \(\mathbb{P}^1\), in order to move \(P\) and \(\phi(P)\) away from \(\infty\)). This is indeed a generalization of quadratic polynomials since the latter are exactly quadratic rational functions with a rational \emph{fixed} critical point (up to conjugation by an automoprhism of \(\mathbb{P}^1\)). In \cite{MR3656200}, Canci and the author provided a full classification of preperiodicity graphs of quadratic rational functions defined over \(\mathbb{Q}\) with a \(\mathbb{Q}\)-rational periodic critical point of period \(2\), up to the following conjecture.

\begin{conjecture} \label{conj:jksoli}
Let $\phi$ be a quadratic rational function defined over $\mathbb{Q}$ having a $\mathbb{Q}$-rational periodic critical point of period $2$. Then $\phi$ has no $\mathbb{Q}$-rational periodic points of period greater than \(2\).
\end{conjecture}

\begin{thm}
Assuming the conjecture above, there are exactly $13$ possible preperiodicity graphs for quadratic rational functions defined over $\mathbb{Q}$ with a $\mathbb{Q}$-rational periodic critical point of period $2$. Moreover, the number of $\mathbb{Q}$-rational preperiodic points of such maps is at most $9$.
\end{thm}

In order to obtain a similar classification for quadratic rational functions with a \(\mathbb{Q}\)-rational periodic critical point of period \(3\), we need to assume a similar conjecture to the conjecture of Flynn, Poonen and Schaefer and Conjecture \ref{conj:jksoli}. First, let us define a \emph{critical \(n\)-cycle} to be the set of iterates of a periodic critical point of period \(n\). 

\begin{conjecture}
\label{conj:main1}
Let $\phi$ be a quadratic rational function defined over $\mathbb{Q}$ with a $\mathbb{Q}$-rational critical \(3\)-cycle. Then $\phi$ has no $\mathbb{Q}$-rational periodic points of period greater than $2$ lying outside of the critical cycle.
\end{conjecture}

As (slight) support for this conjecture, we show in section \ref{support-conj1} that a rational function defined over \(\mathbb{Q}\) with a \(\mathbb{Q}\)-rational critical 3-cycle cannot have a \(\mathbb{Q}\)-rational periodic point of period 3 or 4 lying outside of the critical 3-cycle. For the classification of quadratic rational functions with a critical 3-cycle, we also need to assume the following conjecture.

\begin{conjecture}
\label{conj:main2}
The genus \(6\) affine plane curve defined by the affine equation
\begin{align} 
\begin{split} \label{eq:n3h2}
u^5v^2 + 2u^4v^3 - u^4v^2 &- u^4v + u^3v^4 - 4u^3v^2 + u^3 \\
&+ u^2v^4 - 4u^2v^3 
 + 3u^2v - 2uv^3 + 4uv^2 - u + v^2 - v = 0
\end{split}
\end{align}
has Jacobian of Mordell--Weil rank exactly 2, and the $\mathbb{Q}$-rational points on this curve are exactly
\((-1,0), (1,1), (-1,1), (0,0), (0,1),\) and \( (1,0).\)
\end{conjecture}

In the appendix to this article M. Stoll proves Conjecture \ref{conj:main2}, conditional on standard conjectures including the BSD conjecture, by enumerating all rational points on the curve defined by Eq.\eqref{eq:n3h2} using the Chabauty--Coleman method.

\begin{theorem}
\label{thm:main1}
Assuming Conjectures \ref{conj:main1} and \ref{conj:main2}, there are exactly six possible preperiodicity graphs for quadratic rational functions defined over $\mathbb{Q}$ with a $\mathbb{Q}$-rational critical \(3\)-cycle. Moreover, the number of preperiodic points of such maps is at most eleven.
\end{theorem}

The six possible preperiodicity graphs are listed below in Table \ref{table:r3}, and for each graph we provide an example of a rational function which \emph{realizes} this graph (i.e., whose preperiodicity graph is isomorphic to the given graph). We say that a graph is \emph{realizable} if there exists a rational function that realizes it.

The method by which we prove Theorem \ref{thm:main1} is by showing (Section \ref{section:inadmissibility}) that the graphs in Tables \ref{table:n3e}, \ref{table:n3m} and \ref{table:n3h} below are \emph{inadmissible}, i.e. that there exist no quadratic rational functions defined over \(\mathbb{Q}\) with preperiodicity graphs containing isomorphic subgraphs. We prove in Section \ref{section:sufficiency} that this is sufficient to prove Theorem \ref{thm:main1}. The way we prove inadmissibility of a given graph is by constructing a \emph{dynamical modular curve} whose \(\mathbb{Q}\)-rational points correspond to conjugacy classes (under conjugation by automorphisms of \(\mathbb{P}^1\)) of rational functions defined over \(\mathbb{Q}\) with a rational critical \(3\)-cycle that \emph{admit} the graph (i.e. whose preperiodicity graph contains an isomorphic copy of the graph). We then prove that the dynamical modular curve has no \(\mathbb{Q}\)-rational points.

\quad\newline

\textbf{Acknowledgements.} The author deeply thanks Michael Stoll and Maarten Derickx, without whose help and advice this research would not have been completed. The author would also like to thank Benjamin Collas and Konstantin Jakob for many fruitful discussions during the preparation of this article. Finally the author would like to thank Fabrizio Catanese for his support and encouragement. Research for this article has been supported by the Minerva Foundation and by ERC Advanced Grant "TADMICAMT".

\begin{table}[hb]
\begin{center}
\caption{Realizable quadratic rational functions with a $\mathbb{Q}$-rat.\ periodic critical point of period 3} \label{table:r3}
\begin{tabular}{|c|c|c|}\hline

ID& $\phi(z)$                                           & Preperiodicity graph  
\\ 
\hline \hline 
R3P0 & $\displaystyle\frac{1}{(z-1)^2}$ &  \xygraph{ 
                !{<0cm,0cm>;<2cm,0cm>:<0cm,1cm>::} 
                !{(0.7,0.5) }*+{\bullet_{\infty}}="a" 
                !{(0.7,-0.5) }*+{\bullet_{0}}="b" 
                !{(0,0) }*+{\bullet_{1}}="c" 
                !{(-.75,0) }*+{\bullet_{2}}="d"
                "a":@/^/"b"
                "b":@/^/"c"
                "c":@/^/"a"
                "d":"c"
        } 
        \\ 
 \hline

R3P1 & $\displaystyle\frac{2z^2-z-1}{2z^2}$             &  \xygraph{ 
                !{<0cm,0cm>;<2cm,0cm>:<0cm,1cm>::}              
                !{(0,0) }*+{\bullet_{0}}="a" 
                !{(1,0) }*+{\bullet_{\infty}}="b" 
                !{(0.5,-1) }*+{\bullet_{1}}="c" 
                !{(-0.5,-1) }*+{\bullet_{-1/2}}="d"
                !{(1.5,0        ) }*+{\bullet_{-1}}="e"         
                "a":@/^/"b"
                "b":@/^/"c"
                "c":@/^/"a"
                "d":@/^/"a"
                "e":@/^/"c"       
        } 
        \\ 
 \hline

R3P2 & $\displaystyle\frac{z^2+5z-6}{z^2}$              &  \xygraph{ 
                !{<0cm,0cm>;<2cm,0cm>:<0cm,1cm>::} 
                !{(0,0) }*+{\bullet_{0}}="a" 
                !{(1,0) }*+{\bullet_{\infty}}="b" 
                !{(0.5,-1) }*+{\bullet_{1}}="c" 
                !{(-0.5,-1) }*+{\bullet_{-6}}="d"
                !{(1.5,0        ) }*+{\bullet_{6/5}}="e"                
                !{(2,-1) }*+{\bullet_{3}}="f"
                !{(3,0  ) }*+{\bullet_{2}}="g"                          
                "a":@/^/"b"
                "b":@/^/"c"
                "c":@/^/"a"
                "d":@/^/"a"
                "e":@/^/"c"
                "f":"g"   
                "g":@(r,lu) "g"                                         
        } 
        \\ 
 \hline

R3P3 & $\displaystyle\frac{5z^2-7z+2}{5z^2}$            &  \xygraph{ 
                !{<0cm,0cm>;<2cm,0cm>:<0cm,1cm>::} 
                !{(0,0) }*+{\bullet_{0}}="a" 
                !{(1,0) }*+{\bullet_{\infty}}="b" 
                !{(0.5,-1) }*+{\bullet_{1}}="c" 
                !{(-0.5,-1) }*+{\bullet_{2/5}}="d"
                !{(1.5,0        ) }*+{\bullet_{2/7}}="e"                
                !{(-2,1) }*+{\bullet_{2}}="f"
                !{(-2,-1) }*+{\bullet_{1/3}}="g"                                
                "a":@/^/"b"
                "b":@/^/"c"
                "c":@/^/"a"
                "d":"a"
                "e":@/^/"c"
                "f":"d"   
                "g":"d"                                         
        } 
        \\ 
 \hline

R3P4 & $\displaystyle\frac{3z^2-5z+2}{3z^2}$            &  \xygraph{ 
                !{<0cm,0cm>;<2cm,0cm>:<0cm,1cm>::} 
                !{(0,0) }*+{\bullet_{0}}="a" 
                !{(1,0) }*+{\bullet_{\infty}}="b" 
                !{(0.5,-1) }*+{\bullet_{1}}="c" 
                !{(-0.5,-1) }*+{\bullet_{2/3}}="d"
                !{(1.5,0) }*+{\bullet_{2/5}}="e"                
                !{(-0.5,-2) }*+{\bullet_{-2}}="f"
                !{(0.5,-2) }*+{\bullet_{2}}="g"         
                !{(1.5,-2) }*+{\bullet_{1/3}}="h"
                !{(2.5,-2) }*+{\bullet_{1/2}}="i"               
                "a":@/^/"b"
                "b":@/^/"c"
                "c":@/^/"a"
                "d":@/^/"a"
                "e":@/^/"c"   
                "f":"g"
                        "g":@/^/"h"
                "h":@/^/"g"             
                "i": "h"                                                
        } 
        \\ 
 \hline

R3P5 & $\displaystyle\frac{5z^2-11z+6}{5z^2}$           &  \xygraph{ 
                !{<0cm,0cm>;<2cm,0cm>:<0cm,1cm>::} 
                !{(0.5,0) }*+{\bullet_{0}}="a" 
                !{(1.5,0) }*+{\bullet_{\infty}}="b" 
                !{(1,-1) }*+{\bullet_{1}}="c" 
                !{(0,-1) }*+{\bullet_{6/5}}="d"
                !{(2,0) }*+{\bullet_{6/11}}="e"         
                !{(1,-3) }*+{\bullet_{2/3}}="f"
                !{(2,-3 ) }*+{\bullet_{2/5}}="g"                
                !{(3,-3) }*+{\bullet_{3}}="h"
                !{(4,-3 ) }*+{\bullet_{-3/2}}="i"
                !{(0,-2) }*+{\bullet_{6}}="j"
                !{(0,-4) }*+{\bullet_{3/5}}="k"         
                "a":@/^/"b"
                "b":@/^/"c"
                "c":@/^/"a"
                "d":@/^/"a"
                "e":@/^/"c"   
                "f":"g"
                        "g":@/^/"h"
                "h":@/^/"g"             
                "i": "h"
                "j":"f"
                "k":"f"                         
        } 
        \\ 
 \hline

\end{tabular}
\end{center}

\end{table}
\FloatBarrier

We split the nine inadmissible graphs needed for the proof of Theorem \ref{thm:main1} into three groups, according to the genus of the corresponding dynamical modular curve.

\begin{table}[h]
\begin{center}
\caption{Inadmissible graphs for quadratic rational functions with a critical 3-cycle with a modular curve of genus $1$}
\label{table:n3e}
\begin{tabular}{|c|c|c|}\hline

ID                                              & Preperiodicity graph  & genus\\ 
\hline \hline 

N3E1            &  \xygraph{ 
                !{<0cm,0cm>;<2cm,0cm>:<0cm,1cm>::} 
                !{(0,0) }*+{\bullet_{0}}="a" 
                !{(1,0) }*+{\bullet_{\infty}}="b" 
                !{(0.5,-1) }*+{\bullet_{1}}="c" 
                !{(-0.5,0) }*+{\bullet_{}}="d"
                !{(1.5,0        ) }*+{\bullet_{}}="e"           
                !{(2,-1) }*+{\bullet_{}}="h"
                !{(3,0) }*+{\bullet_{}}="i"                                                                     
                !{(4,-1) }*+{\bullet_{}}="j"
                !{(5,0) }*+{\bullet_{}}="k"                                                                                     
                "a":@/^/"b"
                "b":@/^/"c"
                "c":@/^/"a"
                "d":"a"
                "e":@/^/"c"
                "h":"i"
                "i":@(r,lu) "i"                                                         
                "j":"k"
                "k":@(r,lu) "k"                                                                         
        } & 1 \\ 
\hline

N3E2            &  \xygraph{ 
                !{<0cm,0cm>;<2cm,0cm>:<0cm,1cm>::} 
                !{(0,0) }*+{\bullet_{0}}="a_1" 
                !{(1,0) }*+{\bullet_{\infty}}="b_1" 
                !{(0.5,-1) }*+{\bullet_{1}}="c1" 
                !{(-0.5,-1) }*+{\bullet_{}}="d1"
                !{(1.5,0) }*+{\bullet_{}}="e1"          
                !{(2,-1) }*+{\bullet_{}}="f1"
                !{(2,0) }*+{\bullet_{}}="g1"                            
                "a_1":@/^/"b_1"
                "b_1":@/^/"c1"
                "c1":@/^/"a_1"
                "d1":@/^/"a_1"
                "e1":@/^/"c1"       
                "f1":@/^/"e1"
                "g1":@/^/"e1"                           
        }  & 1\\ 
\hline

N3E3            &  \xygraph{ 
                !{<0cm,0cm>;<2cm,0cm>:<0cm,1cm>::} 
                !{(0,0) }*+{\bullet_{0}}="a" 
                !{(1,0) }*+{\bullet_{\infty}}="b" 
                !{(0.5,-1) }*+{\bullet_{1}}="c" 
                !{(-0.5,0) }*+{\bullet_{}}="d"
                !{(1.5,0        ) }*+{\bullet_{}}="e"           
                !{(-1,1) }*+{\bullet_{}}="f"
                !{(-1,-1) }*+{\bullet_{}}="g"   
                !{(2.5,0) }*+{\bullet_{}}="a_2" 
                !{(3.5,0) }*+{\bullet_{}}="b_2" 
                !{(2,0) }*+{\bullet_{}}="c2" 
                !{(4,0) }*+{\bullet_{}}="d2"                            
                "a":@/^/"b"
                "b":@/^/"c"
                "c":@/^/"a"
                "d":"a"
                "e":@/^/"c"
                "f":"d"   
                "g":"d"
                "a_2":@/^/"b_2"
                "b_2":@/^/"a_2"       
                "c2":"a_2"
                "d2":"b_2"
                } & 1\\ 
\hline
\end{tabular}
\end{center}
\end{table}

\begin{table}[h!]
\begin{center}
\caption{Inadmissible graphs for quadratic rational functions with a critical 3-cycle with a modular curve of genus $2$}
\label{table:n3m}
\begin{tabular}{|c|c|c|}\hline

ID                                              & Preperiodicity graph  & genus\\ 
\hline \hline 

N3M1            &  \xygraph{ 
                !{<0cm,0cm>;<2cm,0cm>:<0cm,1cm>::} 
                !{(0,0) }*+{\bullet_{0}}="a" 
                !{(1,0) }*+{\bullet_{\infty}}="b" 
                !{(0.5,-1) }*+{\bullet_{1}}="c" 
                !{(-0.5,0) }*+{\bullet_{}}="d"
                !{(1.5,0        ) }*+{\bullet_{}}="e"           
                !{(-1,1) }*+{\bullet_{}}="f"
                !{(-1,-1) }*+{\bullet_{}}="g"   
                !{(-1.5,2) }*+{\bullet_{}}="h"
                !{(-1.5,0) }*+{\bullet_{}}="i"                          
                "a":@/^/"b"
                "b":@/^/"c"
                "c":@/^/"a"
                "d":"a"
                "e":@/^/"c"
                "f":"d"   
                "g":"d"                                         
                "h":"f"   
                "i":"f"                                                         
        } & 2\\ 
\hline 

N3M2            &  \xygraph{ 
                !{<0cm,0cm>;<2cm,0cm>:<0cm,1cm>::} 
                !{(0,0) }*+{\bullet_{0}}="a" 
                !{(1,0) }*+{\bullet_{\infty}}="b" 
                !{(0.5,-1) }*+{\bullet_{1}}="c" 
                !{(-0.5,0) }*+{\bullet_{}}="d"
                !{(1.5,0        ) }*+{\bullet_{}}="e"           
                !{(-1,1) }*+{\bullet_{}}="f"
                !{(-1,-1) }*+{\bullet_{}}="g"   
                !{(2,-1) }*+{\bullet_{}}="h"
                !{(3,0) }*+{\bullet_{}}="i"                                                                     
                "a":@/^/"b"
                "b":@/^/"c"
                "c":@/^/"a"
                "d":"a"
                "e":@/^/"c"
                "f":"d"   
                "g":"d"
                "h":"i"
                "i":@(r,lu) "i"                                                         
        } & 2\\ 
\hline

N3M3            &  \xygraph{ 
                !{<0cm,0cm>;<2cm,0cm>:<0cm,1cm>::} 
                !{(0,0) }*+{\bullet_{0}}="a_1" 
                !{(1,0) }*+{\bullet_{\infty}}="b_1" 
                !{(0.5,-1) }*+{\bullet_{1}}="c1" 
                !{(-0.5,-1) }*+{\bullet_{}}="d1"
                !{(1.5,0) }*+{\bullet_{}}="e1"          
                !{(2.5,0) }*+{\bullet_{}}="a_2" 
                !{(3.5,0) }*+{\bullet_{}}="b_2" 
                !{(2,0) }*+{\bullet_{}}="c2" 
                !{(4,0) }*+{\bullet_{}}="d2"            
                !{(2.5,-1) }*+{\bullet_{}}="a3" 
                !{(3.5,-1) }*+{\bullet_{}}="b3"                                 
                "a_1":@/^/"b_1"
                "b_1":@/^/"c1"
                "c1":@/^/"a_1"
                "d1":@/^/"a_1"
                "e1":@/^/"c1"       
                "a_2":@/^/"b_2"
                "b_2":@/^/"a_2"       
                "c2":"a_2"
                "d2":"b_2"
                "a3":"b3"
                "b3":@(r,lu) "b3"                                                                               
        }  & 2\\ 
\hline

\end{tabular}
\end{center}
\end{table}

\begin{table}[h]
\begin{center}
\caption{Inadmissible graphs for quadratic rational functions with a critical 3-cycle with a modular curve of genus $>2$}
\label{table:n3h}
\begin{tabular}{|c|c|c|}\hline
ID                                              & Preperiodicity graph  & genus\\ 
\hline \hline 

N3H1    &  \xygraph{ 
                !{<0cm,0cm>;<2cm,0cm>:<0cm,1cm>::} 
                !{(0,0) }*+{\bullet_{0}}="a_1" 
                !{(1,0) }*+{\bullet_{\infty}}="b_1" 
                !{(0.5,-1) }*+{\bullet_{1}}="c1" 
                !{(-0.5,-1) }*+{\bullet_{}}="d1"
                !{(1.5,0) }*+{\bullet_{}}="e1"          
                !{(2,1) }*+{\bullet_{}}="a_2" 
                !{(3,0) }*+{\bullet_{}}="b_2" 
                !{(2,-1) }*+{\bullet_{}}="c2" 
                !{(4,0) }*+{\bullet_{}}="d2"            
                "a_1":@/^/"b_1"
                "b_1":@/^/"c1"
                "c1":@/^/"a_1"
                "d1":@/^/"a_1"
                "e1":@/^/"c1"       
                "a_2":"b_2"
                "c2":"b_2"
                "b_2":"d2"
                "d2":@(r,lu) "d2"                                                                               
        }  & 3\\ 
\hline

N3H2 &  \xygraph{ 
                !{<0cm,0cm>;<2cm,0cm>:<0cm,1cm>::} 
                !{(0,0) }*+{\bullet_{0}}="a" 
                !{(1,0) }*+{\bullet_{\infty}}="b" 
                !{(0.5,-1) }*+{\bullet_{1}}="c" 
                !{(-0.5,-1) }*+{\bullet_{}}="d"
                !{(1.5,0) }*+{\bullet_{}}="e"           
                !{(1,-3) }*+{\bullet_{}}="f"
                !{(2,-3 ) }*+{\bullet_{}}="g"           
                !{(3,-3) }*+{\bullet_{}}="h"
                !{(4,-3 ) }*+{\bullet_{}}="i"
                !{(0,-2) }*+{\bullet_{}}="j"
                !{(0,-4) }*+{\bullet_{}}="k"
                !{(-1,-2) }*+{\bullet_{}}="l"
                !{(-1,-3) }*+{\bullet_{}}="m"                   
                "a":@/^/"b"
                "b":@/^/"c"
                "c":@/^/"a"
                "d":@/^/"a"
                "e":@/^/"c"   
                "f":"g"
                        "g":@/^/"h"
                "h":@/^/"g"             
                "i": "h"
                "j":"f"
                "k":"f"
                "l":"j"
                "m":"j"                         
        } & 6 \\ 
\hline
N3H3 &  \xygraph{ 
                !{<0cm,0cm>;<2cm,0cm>:<0cm,1cm>::} 
                !{(0,0) }*+{\bullet_{0}}="a" 
                !{(1,0) }*+{\bullet_{\infty}}="b" 
                !{(0.5,-1) }*+{\bullet_{1}}="c" 
                !{(-0.5,-1) }*+{\bullet_{}}="d"
                !{(1.5,0) }*+{\bullet_{}}="e"           
                !{(1,-3) }*+{\bullet_{}}="f"
                !{(2,-3 ) }*+{\bullet_{}}="g"           
                !{(3,-3) }*+{\bullet_{}}="h"
                !{(4,-3 ) }*+{\bullet_{}}="i"
                !{(0,-2) }*+{\bullet_{}}="j"
                !{(0,-4) }*+{\bullet_{}}="k"
                !{(5,-4) }*+{\bullet_{}}="l"
                !{(5,-2) }*+{\bullet_{}}="m"                    
                "a":@/^/"b"
                "b":@/^/"c"
                "c":@/^/"a"
                "d":@/^/"a"
                "e":@/^/"c"   
                "f":"g"
                        "g":@/^/"h"
                "h":@/^/"g"             
                "i": "h"
                "j":"f"
                "k":"f"
                "l":"i"
                "m":"i"                         
        } & 5 \\ 
\hline

\end{tabular}
\end{center}
\end{table}

\FloatBarrier

\section{Preliminaries}
\label{sec:orgheadline5}
\subsection{Dynatomic polynomials}
\label{sec:orgheadline2}
Let \(\phi:\mathbb{P}^1\to\mathbb{P}^1\) be a rational function. We write 
\[\phi = [F(X,Y),G(X,Y)]\]
using homogeneous polynomials \(F\) and \(G\).
Let 
\[\phi^n(X,Y) = [F_n(X,Y), G_n(X,Y)]\] 
be the \(n\)-th iterate of \(\phi\) for \(n\geq{1}\), and define
\[
\Phi_{\phi,n}(X,Y) = YF_n(X,Y) - XG_n(X,Y).
\]
We then define the \emph{\(n\)-th dynatomic polynomial} by
\begin{equation*}
\Phi_{\phi,n}^*(X,Y) = \prod_{k|n} (\Phi_{k,\phi}(X,Y))^{\mu(n/k)},
\end{equation*}
where \(\mu\) is the Moebius mu function (for the proof that \(\Phi_{\phi,n}^*\) are actually polynomials, see \cite[Theorem 2.1]{MR1324210}).

It is easy to see that if \(P\) is a periodic point of period \(n\) then \(P\) is a root of \(\Phi_{\phi,n}^*\). The converse is not true however, and a root of \(\Phi_{\phi,n}^*\) can correspond to a periodic point of period strictly dividing \(n\). For any \(n\geq{1}\), we call the roots of \(\Phi_{\phi,n}^*\) periodic points of \emph{formal period} \(n\) for \(\phi\).

In the article we identify (by abuse of notation) the dynatomic polynomials \(\Phi_{\phi,n}^*(X,Y)\) with their dehomogenization \(\Phi_{\phi,n}^*(z)\), where \(z=\frac{X}{Y}\). By dehomogenizing, we ignore the possibility of \(\infty\) being a root of the dynatomic polynomials. This will not be a problem in what follows, however.

\subsection{Linear equivalence of rational functions}
\label{sec:orgheadline3}
We say that two rational functions \(\phi_1,\phi_2:\mathbb{P}^1\rightarrow\mathbb{P}^1\) are \emph{linearly equivalent} if there exists an \(f\in \text{PGL}_2\) acting as a projective automorphism of \(\mathbb{P}^1\) such that \(\phi_2 = \phi_1^f = f^{-1}\phi_1{f}\). 

Let \(\phi_1, \phi_2\) be linearly equivalent. A point \(P\) is periodic of period \(n\) for \(\phi_1\) if and only if \(f^{-1}(P)\) is periodic of period \(n\) for \(\phi_2\) (similarly for preperiodic points). When \(\phi_1, \phi_2\) and \(f\) are all defined over the same base field \(K\), then it is clear that the preperiodicity graphs of \(\phi_1\) and \(\phi_2\) are isomorphic. Therefore when classifying realizable graphs of rational functions defined over \(\mathbb{Q}\) we are actually interested in the \(\mathbb{Q}\)-rational conjugacy classes of quadratic rational functions rather than in the individual maps.

\subsection{Post-critically finite quadratic rational functions \label{sec-pcf}}
\label{sec:orgheadline4}
A rational function \(\phi:\mathbb{P}^1\to\mathbb{P}^1\) is called \emph{post-critically finite} (or PCF for short) if all of its critical points are preperiodic. A rational function has exactly \(2d-2\) critical points (counted with multiplicity; see \cite{MR2316407}, Section 1.2), and a quadratic rational function has exactly \(2\) \emph{distinct} critical points.

Lukas, Manes and Yap \cite{MR3240812} provided a complete classification of all quadratic post-critically rational functions over \(\mathbb{Q}\) and their \(\bar{\mathbb{Q}}\)-conjugacy classes, as well as all possible preperiodicity graphs of these functions. In particular, they showed that a quadratic rational function defined over \(\mathbb{Q}\) with a \(\mathbb{Q}\)-rational critical \(3\)-cycle is PCF if and only if both of its critical points lie in the same \(3\)-cycle, and there is a unique conjugacy class of quadratic rational functions (over \(\mathbb{Q}\)) with this property, realizing the graph R3P0 in Table \ref{table:r3}.

\section{Sufficiency of the nine inadmissibility graphs \label{section:sufficiency}}
\label{sec:orgheadline6}

\begin{prop} \label{thm:main2}
Assume Conjecture~\ref{conj:main1}. Any quadratic rational function with a rational critical \(3\)-cycle that does not realize one of the graphs in Table \ref{table:r3}, must admit one of the graphs in Tables \ref{table:n3e}, \ref{table:n3m} or \ref{table:n3h}.
\end{prop}

\begin{proof}
We can arrange the graphs from Tables 1 through 4 in a Hasse diagram with respect to the partial order of subgraph isomorphism:

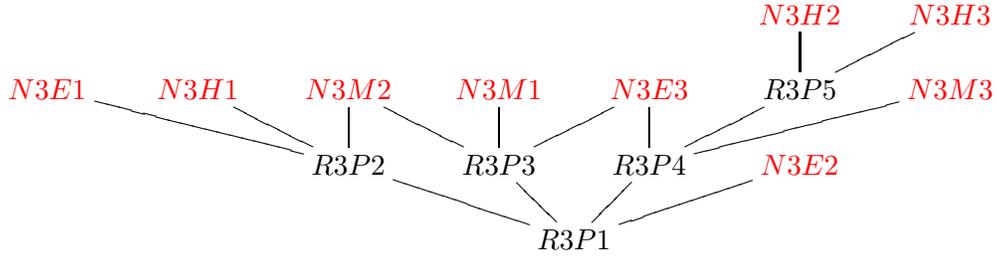
\begin{equation*}
\xygraph{
!{<0cm,0cm>;<1cm,0cm>:<0cm,1cm>::}
!{(0,0) }*+{{R3P1}}="R3P1" 
!{(-3,1) }*+{{R3P2}}="R3P2"  
!{(-1,1) }*+{{R3P3}}="R3P3"
!{(1,1) }*+{{R3P4}}="R3P4"
!{(3,1) }*+{{\color{red}N3E2}}="N3E2"
!{(-7,2) }*+{{\color{red}N3E1}}="N3E1"  
!{(-5,2) }*+{{\color{red}N3H1}}="N3H1"
!{(-3,2) }*+{{\color{red}N3M2}}="N3M2"  
!{(-1,2) }*+{{\color{red}N3M1}}="N3M1"
!{(1,2) }*+{{\color{red}N3E3}}="N3E3"
!{(3,2) }*+{{R3P5}}="R3P5" 
!{(5,2) }*+{{\color{red}N3M3}}="N3M3"
!{(3,3) }*+{{\color{red}N3H2}}="N3H2"
!{(5,3) }*+{{\color{red}N3H3}}="N3H3"
"R3P1"-"R3P2"
"R3P1"-"R3P3"
"R3P1"-"R3P4"
"R3P1"-"N3E2"
"R3P2"-"N3E1"
"R3P2"-"N3H1"
"R3P2"-"N3M2"
"R3P3"-"N3M1"
"R3P3"-"N3M2"
"R3P3"-"N3E3"
"R3P4"-"N3E3"
"R3P4"-"R3P5"
"R3P4"-"N3M3"
"R3P5"-"N3H2"
"R3P5"-"N3H3"
}
\end{equation*}

Due to the results of \cite{MR3240812} mentioned in Section \ref{sec-pcf}, we can restrict ourselves to non-PCF (see Section \ref{sec-pcf}) quadratic rational functions with a rational critical \(3\)-cycle. Lemma 3.2 in \cite{MR3656200} implies that starting from the critical cycle graph R3P1, one can obtain any realizable graph of a non-PCF rational function defined with a rational critical 3-cycle, by recursively adding vertices and arrows to the graph using either of the following two recursion steps:
\begin{enumerate}
\item \textbf{Adding a periodic cycle}. We add a new cycle \(C\) to the graph. For each vertex \(P\) in \(C\) we add a vertex \(Q\) not in the cycle and an arrow \(Q\to{P}\).
\item \textbf{Adding preimages to a non-periodic point}. We add two vertices \(Q_1,Q_2\) and arrows \newline \(Q_1\to{P}, Q_2\to{P}\), where \(P\) is a vertex not lying in any cycle in the graph.
\end{enumerate}

Now it is easy to check that from any graph in the Hasse diagram, by using one of the recursive steps we either generate a graph that is already in the Hasse diagram, or contains (a subgraph isomorphic to) one of the nine graphs in Tables \ref{table:n3e}, \ref{table:n3m} or \ref{table:n3h}.
\end{proof}

\section{Constructing dynamical modular curves from the graphs}
\label{sec:orgheadline27}

\subsection{Realizable graphs with a critical \(3\)-cycle}
\label{sec:orgheadline13}

\subsubsection{R3P0}
\label{sec:orgheadline7}
Note that graph R3P0 implies that we have two points in the 3-cycle which are critical. Since any quadratic rational function has exactly two distinct critical points, this means any function admitting the graph is \emph{post-critically finite} (or PCF for short), i.e. all its critical points are preperiodic. 

As mentioned in Section \ref{sec-pcf}, Lukas, Manes and Yap proved in \cite{MR3240812} that there is only one possible preperiodicity graph for a PCF quadratic rational function with a critical three cycle, which is graph R3P0 in Table 1. Moreover, there is only one \(\mathbb{Q}\)-conjugacy class of quadratic rational functions that has this preperiodicity graph, with the following representative.

\begin{equation}
\phi_0(z) = \frac{1}{(z-1)^2}.
\end{equation}

\subsubsection{R3P1 \label{sec:R3P1}}
\label{sec:orgheadline8}
We can parametrize the non-PCF \(\mathbb{Q}\)-conjugacy classes of quadratic rational functions with a rational critical \(3\)-cycle using the following map.
\begin{equation}
a\mapsto \phi_a(z)=\frac{(a+1)z^2-az-1}{(a+1)z^2}, \quad a \neq 0, -1, -2.
\end{equation}
For the special values \(a=0,-2\) we get representatives of the conjugacy class realizing the graph R3P0.

In fact, given any rational function \(\phi\) defined over \(\mathbb{Q}\) with a \(\mathbb{Q}\)-rational critical 3-cycle, we can conjugate \(\phi\) by an element of \(\text{PGL}_2(\mathbb{Q})\) to move the periodic critical point \(P\) to \(0\), its image to \(\infty\) and \(\phi^2(P)\) to \(1\). One can then check that we obtain a function \(\phi_a\) for some \(a\in\mathbb{Q}\).

\begin{figure}[h]
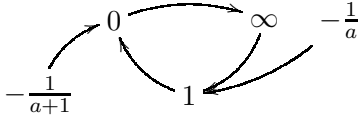

\[
\xygraph{ 
                !{<0cm,0cm>;<2cm,0cm>:<0cm,1cm>::}              
                !{(0,0) }*+{{0}}="a" 
                !{(1,0) }*+{{\infty}}="b" 
                !{(0.5,-1) }*+{{1}}="c" 
                !{(-0.5,-1) }*+{{-\frac{1}{a+1}}}="d"
                !{(1.5,0        ) }*+{{-\frac{1}{a}}}="e"               
                "a":@/^/"b"
                "b":@/^/"c"
                "c":@/^/"a"
                "d":@/^/"a"
                "e":@/^/"c"       
        }
\]
\caption{Preperiodicity graph of \(\phi_a\) \label{fig:R3P1}}
\end{figure}

The preperiodicity graph of \(\phi_a\) where \(a \neq{0,-1,-2}\) will contain the graph in Fig. \ref{fig:R3P1}; it is easy to check that the non-periodic preimages of \(0\) and \(1\) are \(-\frac{1}{a+1}\) and \(-\frac{1}{a}\), respectively.

\subsubsection{R3P2}
\label{sec:orgheadline9}
We take the parametrization \(\phi_a\) of R3P1 and consider a root \(b\) of the first dynatomic polynomial of \(\phi_a\). Such a root must correspond to a fixed point of \(\phi_a\).

\begin{equation}
\Phi_{a,1}^{*}(b) := \Phi_{\phi_a,1}^{*}(b) = (1 + a)b^3 + (-1 - a)b^2 + ab + 1 = 0.
\end{equation}

We solve this equation for \(a\).

\begin{equation} 
a = -\frac{b^{3} - b^{2} + 1}{b(b^{2} - b + 1)}.
\label{eq:a-interms-b}
\end{equation}

We then substitute this into the expression for \(\phi_a\) and get the following parametrization of quadratic rational functions realizing R3P2.

\begin{equation}
b \mapsto \phi_b(z) = \frac{(b-1)z^2 + (b^3 - b^2 + 1)z - b^3 + b^2-b}{(b-1)z^2},
\end{equation}
where \(b\) cannot obtain the values \(0, 1\) in \(\mathbb{Q}\). We remark that \(\phi_b\) degenerates also when \(b\) is a root of \(b^2-b+1 = 0\), but for this case \(b\not\in\mathbb{Q}\). Moreover, the solutions to the equations
\[b^3-b^2+1 = 0 \text{ and } b^3-b^2+2b-1 = 0\]
determine \(a=0\) and \(a=2\), respectively, but these do not have \(\mathbb{Q}\)-rational solutions either.

Each map \(\phi_b\) admits the graph in Fig. \ref{fig:R3P2}.

\begin{figure}[h]
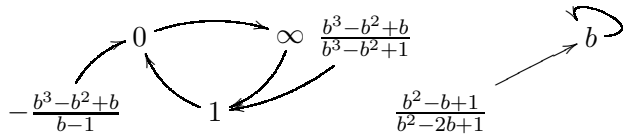

\[
\xygraph{ 
                !{<0cm,0cm>;<2cm,0cm>:<0cm,1cm>::} 
                !{(0,0) }*+{0}="a" 
                !{(1,0) }*+{{\infty}}="b" 
                !{(0.5,-1) }*+{{1}}="c" 
                !{(-0.5,-1) }*+{{-\frac{b^3-b^2+b}{b-1}}}="d"
                !{(1.5,0        ) }*+{{\frac{b^3-b^2+b}{b^3-b^2+1}}}="e"                
                !{(2,-1) }*+{{\frac{b^2-b+1}{b^2-2b+1}}}="f"
                !{(3,0  ) }*+{{b}}="g"                          
                "a":@/^/"b"
                "b":@/^/"c"
                "c":@/^/"a"
                "d":@/^/"a"
                "e":@/^/"c"
                "f":"g"   
                "g":@(r,lu) "g"                                         
        }
\]
\caption{Preperiodicity graph of \(\phi_b\) \label{fig:R3P2}}
\end{figure}

It is easy to check that
\[-\frac{b^3-b^2+b}{b-1}, \quad \frac{b^3-b^2+b}{b^3-b^2+1} \text{ and } \quad \frac{b^2-b+1}{(b-1)^2}\]
are the non-periodic preimages of \(0,1\) and \(b\), respectively.

\subsubsection{R3P3}
\label{sec:orgheadline10}
We start again with \(\phi_a\) parametrizing R3P1 (see Section \ref{sec:R3P1}). Recall that \(-1/(a+1)\) is the non-periodic preimage of \(0\). Let \(c\) be a preimage of \(-1/(a+1)\), i.e. \(\phi(c)=-\frac{1}{a+1}\). This implies the following equation.

\begin{equation}
(a+2)c^2 - ac - 1 = 0.
\end{equation}

We solve this equation for \(a\).

\begin{equation}
a = -\frac{2 \, c^{2} - 1}{c^{2} - c}.
\end{equation}

By substituting this expression for \(a\), we get the following parametrization.

\begin{equation}
c \mapsto \phi_c(z) = \frac{{\left(c^{2} + c - 1\right)} z^{2} - {\left(2 \, c^{2} - 1\right)} z + c^{2} - c}{{\left(c^{2} + c - 1\right)} z^{2}},
\end{equation}
where \(c\) cannot obtain the values \(0, 1\) and \(\frac{1}{2}\) (at the latter value we get \(a=-2\)). Moreover, \(c\) cannot be a solution to the equation
\[c^2+c-1 = 0,\]
as \(\phi_c\) degenerates for these values of \(c\); this equation however has no rational solutions. Similarly, for \(c^2=-\frac{1}{2}\) we get \(a = 0\), but again there are no rational \(c\) satisfying this condition.

\begin{figure}[h]
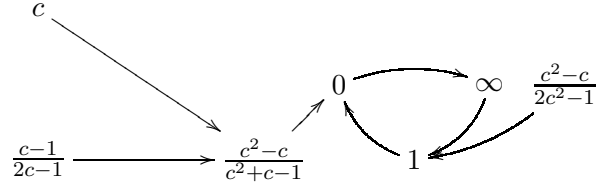

\[
\xygraph{ 
                !{<0cm,0cm>;<2cm,0cm>:<0cm,1cm>::} 
                !{(0,0) }*+{0}="a" 
                !{(1,0) }*+{\infty}="b" 
                !{(0.5,-1) }*+{1}="c" 
                !{(-0.5,-1) }*+{\frac{c^2-c}{c^2+c-1}}="d"
                !{(1.5,0        ) }*+{\frac{c^2-c}{2c^2-1}}="e"         
                !{(-2,1) }*+{c}="f"
                !{(-2,-1) }*+{\frac{c-1}{2c-1}}="g"                             
                "a":@/^/"b"
                "b":@/^/"c"
                "c":@/^/"a"
                "d":"a"
                "e":@/^/"c"
                "f":"d"   
                "g":"d"                                         
        }
\]
\caption{Preperiodicity graph of \(\phi_c\) \label{fig:R3P3}}
\end{figure}

Each map \(\phi_c\) admits the graph in Fig. \ref{fig:R3P3}. It is easy to check that \(\frac{c-1}{2c-1}\) is the second preimage of \(-\frac{1}{a+1} = \frac{c^2-c}{c^2+c-1}\), other than \(c\).

\subsubsection{R3P4}
\label{sec:orgheadline11}
We start with \(\phi_a\) realizing R3P1 and find the expression for the second dynatomic polynomial of \(\phi_a\). A root of this polynomial is a point of formal period \(2\) for \(\phi_a\).

\begin{equation}
\Phi^*_{a,2}(z) := \Phi^*_{\phi_a,2}(z) = (a+1)((a+1)z^2 + (1 - a)z - 1)
\end{equation}

A formal point of period \(2\) for \(\phi_a\) is not of period \(2\) only if it is of period \(1\) and so a root of the first dynatomic polynomial \(\Phi^*_{a,1}(z)\). One can find such points by computing the roots of the resultant of the first two dynatomic polynomials.

\[ \textrm{Res}(\Phi^*_{a,1}(z), \Phi^*_{a,2}(z)) = -(a+1)^4(a^2+2a+5)\]
We see that there are no \(\mathbb{Q}\)-rational values of \(a\) which produce points of formal period \(2\) but of actual period \(1\). 

We see that any \(\mathbb{Q}\)-rational periodic point \(d\) is of period \(2\) for \(\phi_a\) if and only if it is a solution of the following equation.

\begin{equation}
(a+1)d^2 + (1 - a)d - 1 = 0.
\end{equation}

We solve this equation for \(a\).

\begin{equation}
a = -\frac{d^{2} + d - 1}{d^{2} - d}.
\end{equation}

We get the following parametrization for quadratic rational functions admitting the graph R3P4.

\begin{equation}
d \mapsto \phi_d(z)=\frac{(2d - 1)z^2 - (d^2 + d - 1)z + d^2 - d}{(2d - 1)z^2},
\end{equation}
where \(d\) cannot obtain the values \(0,1\) and \(\frac{1}{2}\). When \(d\) is a solution to the equations
\[d^2+d-1 = 0 \quad \text{ and } d^2-3d+1 = 0\]
we get \(a=0\) and \(a=-2\) respectively; these equations have no rational solutions, however.

One can check that the non-periodic preimage of \(d\) is \(-\frac{d}{d-1}\) and its periodic image is \(\frac{d-1}{2d-1}\). The non-periodic preimage of \(\frac{d-1}{2d-1}\) is \(\frac{d-1}{d}\). The non-periodic preimage of \(0\) is \(\frac{d^2-d}{2d-1}\) and the non-periodic preimage of \(1\) is \(\frac{d^2-d}{d^2+d+1}\). Thus the rational function \(\phi_d\) admits the graph in Fig. \ref{fig:R3P4}.

\begin{figure}[h]
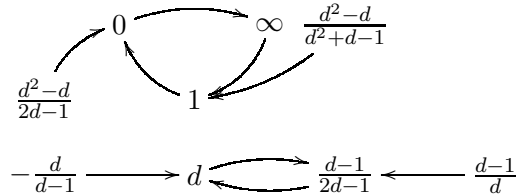

\[
\xygraph{ 
                !{<0cm,0cm>;<2cm,0cm>:<0cm,1cm>::} 
                !{(0,0) }*+{0}="a" 
                !{(1,0) }*+{\infty}="b" 
                !{(0.5,-1) }*+{1}="c" 
                !{(-0.5,-1) }*+{\frac{d^2-d}{2d-1}}="d"
                !{(1.5,0) }*+{\frac{d^2-d}{d^2+d-1}}="e"                
                !{(-0.5,-2) }*+{-\frac{d}{d-1}}="f"
                !{(0.5,-2) }*+{d}="g"           
                !{(1.5,-2) }*+{\frac{d-1}{2d-1}}="h"
                !{(2.5,-2) }*+{\frac{d-1}{d}}="i"               
                "a":@/^/"b"
                "b":@/^/"c"
                "c":@/^/"a"
                "d":@/^/"a"
                "e":@/^/"c"   
                "f":"g"
                        "g":@/^/"h"
                "h":@/^/"g"             
                "i": "h"                                                
        }
\]
\caption{Preperiodicity graph of \(\phi_c\) \label{fig:R3P4}}
\end{figure}

\subsubsection{R3P5 \label{sec-R3P5}}
\label{sec:orgheadline12}
We start with \(\phi_d\) realizing R3P4. Recall that \(d\) is a periodic point of period \(2\) and the value of its non-periodic preimage is \(-\frac{d}{d-1}\).

Let \(w\) be a preimage of \(-\frac{d}{d-1}\). Then \(\phi(w)=-\frac{d}{d-1}\). This gives us the following equation.

\begin{equation}
(4d^2 - 4d + 1)w^2 - (d^3 - 2d + 1)w + d^3-2d^2+d = 0.
\label{eq:R3P5}
\end{equation}

There are four solutions to this equation which correspond to degenerate maps \(\phi_d\), these are
\[(d,w) = (0,0),(0,1), (1,0), (\frac{1}{2}, 1). \]

The discriminant of Eq. \eqref{eq:R3P5} with respect to \(w\) is:

\begin{align*}
\Delta &= d^6 - 16d^5 + 44d^4 - 50d^3 + 28d^2 - 8d + 1 \\ 
&= (d-1)^2(d^4 - 14d^3 + 15d^2 - 6d + 1)
\end{align*}

Using the discriminant, we find that the affine plane curve defined by Eq. \eqref{eq:R3P5} is birational to the affine plane curve defined by

\begin{equation}
y^2 = x^4 - 14x^3 + 15x^2 - 6x + 1,
\label{eq:R3P5b}
\end{equation}

where 
\begin{equation}
x = d, \quad y = \frac{2(4d^2 - 4d + 1)w - (d^3 - 2d + 1)}{d-1}
\end{equation}

or

\begin{equation}
w = \frac{(d-1)y + (d^3 - 2d + 1)}{2(4d^2 - 4d + 1)}.
\end{equation}

We can simplify the curve defined by Eq. \eqref{eq:R3P5b} further to get the elliptic curve

\begin{equation}
E: y^2 + xy + y = x^3 - x^2.
\end{equation}

This elliptic curve is curve 53a1 in the Cremona database (see \cite{MR2282912}), it has rank \(1\) with trivial torsion. We thus get infinitely many rational points on the curve, and thus infinitely many non-linearly equivalent rational functions defined over \(\mathbb{Q}\) admitting graph R3P5 (only finitely many \(\mathbb{Q}\)-rational points on \(E\) will be mapped to degenerate solutions of Eq. \eqref{eq:R3P5}).

\subsection{Nonrealizable graphs with a critical 3-cycle \label{section:inadmissibility}}
\label{sec:orgheadline23}

\subsubsection{N3E1 \label{sec-N3E1}}
\label{sec:orgheadline14}
We start with the graph R3P1 by taking the parametrization \(\phi_a\) we defined in Section \ref{sec:R3P1}. We recall that the general form of a quadratic rational function admitting R3P1 is

\begin{equation}
\phi_a(z)=\frac{(a+1)z^2-az-1}{(a+1)z^2}, \quad a\neq{0,-1,-2}.
\end{equation}

We will show that the first dynatomic polynomial of \(\phi_a\) cannot split over \(\mathbb{Q}\). This immediately implies that a quadratic rational function with a rational critical three-cycle can have at most one \(\mathbb{Q}\)-rational fixed point.

The first dynatomic polynomial of \(\phi_a\) is

\begin{equation}
\Phi^*_{a,1}(z) := \Phi^*_{\phi_a,1}(z) = (-a - 1)z^3 + (a + 1)z^2 -az - 1.
\end{equation}

In R3P2 we determined that a root \(b\) of this polynomial must satisfy
\begin{equation} \label{eq:N3E1-recomputing-phib}
a = -\frac{b^{3} - b^{2} + 1}{b(b^{2} - b + 1)}.
\end{equation}

A necessary and sufficient condition for the cubic polynomial \(\Phi^*_{a,1}\) to split over \(\mathbb{Q}\) is for it to have a \(\mathbb{Q}\)-rational root and for the discriminant to be a square in \(\mathbb{Q}\). The latter condition is defined by the following equation.

\begin{equation} \label{eq:N3E1-discrim}
u^2 = \Delta(\Phi^*_{a,1}) = -3a^4 - 16a^3 - 50 a^2 - 60a - 23.
\end{equation}

Substituting \eqref{eq:N3E1-recomputing-phib} into \eqref{eq:N3E1-discrim} we obtain 
\begin{equation} \label{eq:N3E1-eq}
v^2 = (b-1)(b^3+b^2-b+3),
\end{equation}
where
\begin{equation*}
v = \frac{b^2(b^2-b+1)^2u}{2b^3-4b^2+2b-1}.
\end{equation*}

The affine plane curve \(C\) defined by \eqref{eq:N3E1-eq} is of genus \(1\). We denote by \(\bar{C}\) its projective closure in \(\mathbb{P}^2_{[X:Y:Z]}\) where \(b = X/Z\) and \(v = Y/Z\). The curve \(\bar{C}\) contains (at least) the following two \(\mathbb{Q}\)-rational points:
\[[0:1:0], [1:0:1].\]
It is birational to the elliptic curve with Cremona reference 19a3 with minimal model
\[E : y^2 + y = x^3+x^2+x.\]

\begin{remark}
It is interesting to note that the elliptic curve with Cremona reference 19a3 appears several times in the classification of preperiodicity graphs of rational functions with a rational critical cycle. For example, in \cite{MR3656200} it is proven that it is birational to two modular curves of quite distinct graphs with critical \(2\)-cycles (denoted there by N2E2 and N2E3), and it appears again below in Section \ref{sec-4cycle}. Similarly, the elliptic curve with Cremona reference 17a4 appears both for N3E3 below (see Section \ref{sec-N3E3} below) and for the proof of Proposition 4.3 in \cite{MR3656200}. The elliptic curve with Cremona reference 11a3 appears for N3E2 below (see Section \ref{sec-N3E2} below) and for the graph N2E1 in \cite{MR3656200}.
\end{remark}

The birational map \(\psi:\bar{C}\to{E}\) is given by
\begin{align*}
\psi([X:Y:Z]) = [&2Z^2-2XZ, X^2-2XZ+YZ+Z^2, -2X^2+4XZ-2Z^2].
\end{align*}

The rational points in the locus of indeterminacy of \(\psi\) are exactly the two points we have already discovered on \(\bar{C}\). The elliptic curve \(E\) has Mordell--Weil group of rank 0 and torsion subgroup  isomorphic to \(\mathbb{Z}/3\mathbb{Z}.\) It therefore contains only three rational points, which are
\[[0 : 0 : 1], [0 : 1 : 0], [0 : -1 : 1].\]
The preimages of these points under the birational map \(\psi\), together with \(\psi\)'s locus of indeterminacy give us the full set of points on \(\bar{C}\), and it turns out that the two points we found on \(\bar{C}\) are the only rational points on it. Neither of these points corresponds to a non-degenerate quadratic rational function \(\phi_{\text{b}}\).

As an aside, we prove the following.

\begin{prop}
There exists a unique conjugacy class of quadratic rational functions with a critical 3-cycle whose first dynatomic polynomial has a square \(\mathbb{Q}\)-rational discriminant.
\end{prop}

\begin{proof}
The affine plane curve \(C_1\) described by Eq.\eqref{eq:N3E1-discrim} has genus \(1\), and it is birational to the elliptic curve with Cremona reference 19a1 with the following minimal model.

\begin{equation}
E_1: y^2 + y = x^3 + x^2 - 9x - 15.
\end{equation}

We consider the projective closure \(\bar{C}_1\) of \(C_1\) in \(\mathbb{P}^2_{[X:Y:Z]}\) where \(a = \frac{X}{Z}\) and \(u = \frac{Y}{Z}\).
The birational map \(\psi:\bar{C}_1\to{E_1}\) is given by

\begin{align*}
\psi = [
&1491X^3 + 4176X^2Z - 133XYZ + 3821XZ^2 - 114YZ^2 + 1146Z^3,\\
&533X^3 + 2839X^2Z + 171XYZ + 3891XZ^2 + 95YZ^2 + 1565Z^3,\\
&-686X^3 - 1764X^2Z - 1512XZ^2 - 432Z^3
].
\end{align*}

The elliptic curve \(E_1\) has Mordell--Weil rank 0, and torsion subgroup isomorphic to \(\mathbb{Z}/3\mathbb{Z}\). The three rational points on \(E\) are 

\begin{equation}
[0 : 1 : 0], [5 : 9 : 1], [5 : -10 : 1].
\end{equation}

The only rational preimages of these points under \(\psi\), together with \(\psi\)'s points of indeterminacy are the four points

\begin{equation}
[-6/7 : -19/49 : 1], [-6/7 : 19/49 : 1], [0 : 1 : 0], [-1 : 0 : 1].
\end{equation}

These are therefore the only rational points on the curve \(\bar{C}_1\). The third point is at infinity and is therefore a "\emph{cusp}" (a point added to the dynamical modular curve to obtain a projective curve), while the fourth point corresponds to the case \(a=-1\) for which \(\phi_a\) is a degenerate map. 

It remains to check the case \(a = -\frac{6}{7}\). For this map we get the following first dynatomic polynomial
\begin{equation}
\Phi_{-\frac{6}{7},1}^*(z)=-\frac{1}{7}(z^3-z^2-6z+7)
\end{equation}
whose discriminant \(\frac{361}{2401} = (\frac{19}{49})^2\) is a square in \(\mathbb{Q}\) as expected. However, this polynomial is irreducible and therefore \(\phi_{-\frac{6}{7}}\) has no \(\mathbb{Q}\)-rational fixed points.
\end{proof}

\subsubsection{N3E2 \label{sec-N3E2}}
\label{sec:orgheadline15}

We start with the parametrization \(\phi_a\) of R3P1. Recall that the non-periodic preimage of \(1\) is \(-\frac{1}{a}\). Assume that \(-\frac{1}{a}\) has a preimage \(t\), i.e. \(\phi(t)= -\frac{1}{a}\). This implies the following equation.

\begin{equation}
a^{2} t^{2} - a^{2} t + 2a t^{2} + t^{2} - a = 0.
\end{equation}

The curve \(C_2\) described by this equation is of genus \(1\). We take the projective closure \(\bar{C}_2\) of \(C_2\) in \(\mathbb{P}^2_{[X:Y:Z]}\) where \(a = \frac{X}{Z}\) and \(t = \frac{Y}{Z}\). \(\bar{C}_2\) contains (at least) the following 4 rational points:

\begin{equation}
[0 : 1 : 0], [0 : 0 : 1], [1 : 0 : 0], [-1 : 1 : 1].
\end{equation}

The affine plane curve \(\bar{C}_2\) is birational to the elliptic curve \(E\) with reference 11a3 in the Cremona database. Its minimal model is given by
\begin{equation}
E: y^2 + y = x^3 - x^2.
\end{equation}

This elliptic curve has Mordell--Weil rank \(0\) and a torsion subgroup isomorphic to \(\mathbb{Z}/5\mathbb{Z}\) and \(5\) rational points:

\begin{equation}
[0 : 1 : 0], [0 : 0 : 1], [0 : -1 : 1], [1 : 0 : 1], [1 : -1 : 1].
\end{equation}

The birational map \(\psi:\bar{C}_2\to{E}\) is defined by

\begin{align*}
\psi = [&
XY^3 - XY^2Z + Y^3Z - YZ^3, \\
&-XY^2Z + Y^3Z + XYZ^2 - 2Y^2Z^2 + Z^4, \\
&-Y^3Z].
\end{align*}

The preimages of the rational points under \(\psi\) together with \(\psi\)'s points of indeterminacy (\([1:0:0], [0:1:0]\)) are the four points of \(\bar{C}_2\) we have already discovered. Therefore these are all the rational points on \(\bar{C}_2\).

None of these points corresponds to a non-degenerate map \(\phi_{\text{a}}\), therefore graph N3E2 is non-realizable.

\subsubsection{N3E3 \label{sec-N3E3}}
\label{sec:orgheadline16}

We start the quadratic rational function \(\phi_c\) parametrizing R3P3 (note that we could also have started with R3P4).

\begin{equation}
\phi_c(z) = \frac{{\left(c^{2} + c - 1\right)} z^{2} - {\left(2 \, c^{2} - 1\right)} z + c^{2} - c}{{\left(c^{2} + c - 1\right)} z^{2}}, \quad c\neq 0, 1, \frac{1}{2}.
\end{equation}

We consider the second dynatomic polynomial for \(\phi_c\).

\begin{align*}
\Phi_{c,2}^*(z) := \Phi_{\phi_c,2}^*(z) = &(c^{4} + 2 c^{3} - c^{2} - 2 c + 1) z^{2} \\
& + (-3 c^{4} - 2 c^{3} + 5 c^{2} - 1) z + c^{4} - 2 c^{2} + c
\end{align*}

This polynomial is divisible by \(c^2+c-1\); since the map \(\phi_c\) degenerates for the two values of \(c\) which are its roots, we can divide by it to get
\begin{equation}
c^2z^2 - 3c^2z + c^2 + cz^2 + cz - c - z^2 + z = 0.
\end{equation}

This equation describes a genus \(1\) affine plane curve which we will denote by \(C_3\), and its projective closure in \(\mathbb{P}^2_{[X:Y:Z]}\) by \(\bar{C}_3\), where \(c = \frac{X}{Z}\) and \(z = \frac{Y}{Z}\). The curve \(\bar{C}_3\) contains (at least) the following 6 points:

\begin{equation}
[1 : 1 : 1], [0 : 1 : 0], [0 : 0 : 1], [0 : 1 : 1], [1 : 0 : 0], [1 : 0 : 1].
\end{equation}

This curve is birational to the elliptic curve

\begin{equation}
E : y^2 + xy + y = x^3 - x^2 - x
\end{equation}
with Cremona reference 17a4. This curve has Mordell-Weil rank 0 and torsion subgroup isomorphic to \(\mathbb{Z}/4\mathbb{Z}\), and thus contains exactly four \(\mathbb{Q}\)-rational points which are
\[[0:0:1],[0:1:0],[0:-1:1],[1:-1:1].\] 

The map \(\psi\) from \(\bar{C}_3\) to \(E\) is defined by

\begin{align*}
\psi = [&XY^3 - 4XY^2Z + 4XYZ^2 + Y^2Z^2 - XZ^3 - 
    2YZ^3 + Z^4, \\
&XY^2Z - 3XYZ^2 + XZ^3 + YZ^3 - Z^4, \\
&-Y^3Z + 3Y^2Z^2 - 3YZ^3 + Z^4]. \\
\end{align*}

The points of indeterminacy of \(\psi\) are \([0 : 1 : 1], [0 : 1 : 0], [1 : 0 : 0]\). 

The preimages of the four rational points of \(E\) under \(\psi\) together with \(\psi\)'s points of indeterminacy give us all the rational points of \(\bar{C}_3\), and these are exactly the six points that we have already discovered, none of which correspond to non-degenerate quadratic rational functions.

\subsubsection{N3M1}
\label{sec:orgheadline17}

We start with the parametrization \(\phi_c\) of the graph R3P3. Recall that \(c\) is a preperiodic point whose image under \(\phi_c\) is the non-periodic preimage of \(0\). Let \(w\) be a preimage of \(c\) under \(\phi_c\), i.e. \(\phi(w)=c\). This implies
\begin{equation}
-c^3w^2 - 2c^2w + 2cw^2 + c^2 - w^2 - c + w = 0
\end{equation}

This equation describes an affine plane curve that we denote by \(C_4\), and we denote its projective closure in \(\mathbb{P}^2_{[X:Y:Z]}\) by \(\bar{C}_4\), where \(c = \frac{X}{Z}\) and \(w = \frac{Y}{Z}\). This curve is of genus \(2\) and contains (at least) the following five points.

\begin{equation}
[0 : 1 : 0], [0 : 0 : 1], [0 : 1 : 1], [1 : 0 : 0], [1 : 0 : 1].
\end{equation}

The curve \(\bar{C}_4\) is birational to the hyperelliptic curve given by

\begin{equation}
H: y^2 = x^6 - 2x^5 + 3x^4 - 4x^3 - x^2 + 2x + 1.
\end{equation}

The birational map \(\psi:\bar{C}_4\to{H}\) is defined by

\begin{align*}
\psi = [-XZ + Z^2, 2X^3YZ^2 + 2X^2Z^4 - 4XYZ^4 + 2YZ^5 - Z^6, -XZ].
\end{align*}

The latter curve \(H\) contains the following points (\emph{in weighted homogeneous coordinates}):

\begin{equation}
[1 : -1 : 0], [1 : 1 : 0], [0 : -1 : 1], [0 : 1 : 1], [1 : 0 : 1].
\end{equation}

The hyperelliptic curve has Jacobian with torsion subgroup isomorphic to \(\mathbb{Z}/15\mathbb{Z}\), and a rank bound computation in \textsc{Magma} \cite{MR1484478} tells us the Jacobian has Mordell--Weil rank 0. 

The curve \(H\) has good reduction at \(3\), and therefore the Mordell-Weil group of the Jacobian \(J(\mathbb{Q})\) of \(H\) injects into \(J(\mathbb{F}_3)\). Therefore all the maps in the following commutative diagram are injective.

\begin{center}
\begin{tikzcd}
H(\mathbb{Q}) \arrow{d} \arrow{r}
&J(\mathbb{Q}) \arrow{d}\\
H(\mathbb{F}_3) \arrow{r} & J(\mathbb{F}_3)
\end{tikzcd}
\end{center}

However, one can easily check that the only rational points in \(H(\mathbb{F}_3)\) are the reductions \(\pmod{3}\) of the five points in \(H(\mathbb{Q})\) we have already discovered. Therefore these five points are the only \(\mathbb{Q}\)-rational points on \(H\). 

Using \(\psi\) we pull back these five points to \(C_4\), and together with the indeterminacy points of \(\psi\) (these are \([0 : 1 : 0], [1 : 0 : 0]\)) we find that the only rational points on \(C\) are the five points we found before. None of these points correspond to non-degenerate quadratic rational functions.

\subsubsection{N3M2}
\label{sec:orgheadline18}

We start with the parametrization \(\phi_c\) of the graph R3P3, and consider a root \(w\) of its first dynatomic polynomial. Such a root must correspond to a fixed point of \(\phi_c\). 

\begin{align*}
\Phi_{c,1}^*(w) := \Phi_{\phi_c,1}^*(w) = &(-c^2 - c + 1)w^3 + (c^2 + c - 1)w^2  \\
& + (-2c^2 + 1)w + c^2 - c = 0.
\end{align*}

We denote by \(C_5\) the affine plane curve defined by this equation, and by \(\bar{C}_5\) its projective closure in \(\mathbb{P}^2_{[X:Y:Z]}\), where \(c = \frac{X}{Z}\) and \(w = \frac{Y}{Z}\). This is a curve of genus 2 and contains (at least) the following four rational points
\begin{equation}
[0 : 1 : 0], [0 : 0 : 1], [1 : 0 : 0], [1 : 0 : 1].
\end{equation}
This curve is birational to the following hyperelliptic curve

\begin{equation}
H_5: \quad y^2 = x^6 - 2x^5 + 5x^4 - 6x^3 + 10x^2 - 8x + 5.
\end{equation}

The birational map \(\psi:\bar{C}_5\to{H_5}\) is defined by

\begin{align*}
\psi = [&-YZ + Z^2, \\
&2XY^3Z^2 - 2XY^2Z^3 + Y^3Z^3 + 4XYZ^4 - 
    Y^2Z^4 - 2XZ^5 + Z^6, \\
&-YZ].
\end{align*}

One can use \textsc{Magma} to check that the Jacobian \(J\) of \(H_5\) has Mordell--Weil torsion subgroup isomorphic to \(\mathbb{Z}/5\mathbb{Z}\) and rank 0. The curve \(H_5\) contains (at least) the following two rational points (in weighted homogeneous coordinates).
\begin{equation}
[1 : -1 : 0], [1 : 1 : 0].
\end{equation}

The Jacobian \(J\) contains (at least) the following rational points in Mumford representation:

\begin{align*}
&(1, 0, 0), (1, x^3 - x^2, 2), (1, -x^3 + x^2, 2), (x^2 - x + 1, x - 1, 2), 
(x^2 - x + 1, -x + 1, 2).
\end{align*}

We denote by \(P_+ = [1 : 1 : 0]\) and \(P_- = [1 : -1 : 0]\), the two points at infinity. Using this notation we get the following divisor representations on \(H\):
\begin{align*}
(1,0,0) &\equiv identity, \\
(1, x^3 - x^2, 2) &\equiv P_+ - P_-, \\
(1, -x^3 + x^2, 2) &\equiv P_- - P_+, \\
(x^2 - x + 1, x - 1, 2) &\equiv [\alpha:-\bar\alpha:1] + 
[\bar\alpha:-\alpha:1]
- P_+ - P_-,
 \\
(x^2 - x + 1, -x + 1, 2) &\equiv
[\alpha:\bar\alpha:1] + [\bar\alpha:\alpha:1]
- P_+ - P_-,
\end{align*}
where \(\alpha=\frac{1}{2}+\frac{\sqrt{-3}}{2}\).
The only rational points appearing in these representations are the two we already found,
and therefore they are the only rational points on \(H_5\).

Pulling back the two rational points of \(H_5\) under \(\psi\) and adding its indeterminacy points, we find all rational points on \(C_5\), and these are exactly the four points we found before. None of these points correspond to non-degenerate quadratic rational functions.

\subsubsection{N3M3}
\label{sec:orgheadline19}

We start with R3P4 and consider a root \(w\) of first dynatomic polynomial of \(\phi_d\) (note that we can also start with R3P2 and look at the second dynatomic polynomial of \(\phi_b\)). Such a root corresponds to a fixed point of \(\phi_d\). 

\begin{align*}
\Phi_{d,1}^*(w) := \Phi_{d,1}^*(w) = &(-2d + 1)w^3 + (2d - 1)w^2 + (-d^2 - d + 1)w + d^2 - d=0
\end{align*}

The affine plane curve \(C_6\) described by this equation is of genus \(2\). We denote its projective closure in \(\mathbb{P}^2_{[X:Y:Z]}\) by \(\bar{C}_6\), where \(d = \frac{X}{Z}\) and \(w = \frac{Y}{Z}\). The curve \(\bar{C}_6\) contains (at least) the following five rational points
\begin{equation}
[0 : 1 : 0], [0 : 0 : 1], [1 : 2 : 2], [1 : 0 : 0], [1 : 0 : 1],
\end{equation}
and is birational to the hyperelliptic curve

\begin{equation}
H: y^2 = x^6 + 2x^5 + 5x^4 + 8x^3 + 12x^2 + 8x + 4
\end{equation}
under the map \(\psi:\bar{C}_6\to{H}\) defined by
\begin{align*}
\psi = [Z, 2Y^3 + 2XYZ - 2Y^2Z - 2XZ^2 + YZ^2 + Z^3, -Y].
\end{align*}

Using \textsc{Magma} one can compute that the hyperelliptic curve \(H\) has Jacobian with Mordell--Weil torsion subgroup isomorphic to \(\mathbb{Z}/2\mathbb{Z}\) and rank bounded by 1. The curve \(H\) contains (at least) the following six rational points

\begin{equation}
[1 : -1 : 0], [1 : 1 : 0], [-1 : -2 : 1], [-1 : 2 : 1], [0 : -2 : 1], [0 : 2 : 1].
\end{equation}

We denote by \(P_+ = [1 : 1 : 0]\) and \(P_- = [1 : -1 : 0]\), the two points at infinity. One can check that \(P_+ - P_-\) is of infinite order, and therefore the Mordell-Weil rank of the Jacobian is exactly 1. Using the Chabauty--Coleman/ Mordell--Weil sieving algorithm (see Bruin and Stoll \cite{MR2685127}) implementation in \textsc{Magma}, we determine that the six points we have already found are the only rational points on \(H\). Pulling them back to \(\bar{C}_6\), we find (together with the indeterminacy point \([1:0:0]\) of \(\psi\)) that the only rational points on \(\bar{C}_6\) are the five points we have already found. None of these correspond to non-degenerate quadratic rational functions.

\subsubsection{N3H1}
\label{sec:orgheadline20}

We start with R3P2. Recall that

\begin{equation}
\phi_b(z) = \frac{(b-1)z^2 + (b^3 - b^2 + 1)z - b^3 + b^2-b}{(b-1)z^2}, \quad b\neq{0,1}
\end{equation}
is a parametrization of R3P2, where \(b\) is the fixed point and \(\frac{b^2-b+1}{(b-1)^2}\) is its non-periodic preimage. Assume \(w\) is a preimage of \(\frac{b^2-b+1}{(b-1)^2}\), i.e. \(\phi_b(w)=\frac{b^2-b+1}{(b-1)^2}\). This implies

\begin{equation}
bw^2 - (b^4 - 2b^3 + b^2 + b - 1)w + b^4 - 2b^3 + 2b^2 - b = 0.
\end{equation}

We denote by \(C_7\) the affine plane curve defined by this equation, and by \(\bar{C}_7\) its projective closure in \(\mathbb{P}^2_{[X:Y:Z]}\), where \(b = \frac{X}{Z}\) and \(w = \frac{Y}{Z}\). The curve \(\bar{C}_7\) has genus 3, and contains (at least) the following rational points

\begin{equation}
[0 : 1 : 0], [0 : 0 : 1], [1 : 0 : 0], [1 : 0 : 1].
\end{equation}

The curve \(\bar{C}_7\) is birational to the following hyperelliptic curve

\begin{equation}
H: y^2 = 4x^7 - 11x^6 + 14x^5 - 7x^4 - 2x^3 + 6x^2 - 4x + 1.
\end{equation}

The birational map \(\psi:\bar{C}_7\to{H}\) is defined by:

\begin{align*}
\psi = [-Z, X^4 - 2X^3Z + X^2Z^2 - 2XYZ^2 + XZ^3 - Z^4, X - Z].
\end{align*}

The curve \(H\) contains (at least) the following five points:

\begin{equation}
[1 : 0 : 0], [0 : -1 : 1], [0 : 1 : 1], [1 : -1 : 1], [1 : 1 : 1].
\end{equation}

Using \textsc{Magma}, one can compute that the curve \(H\) has Jacobian \(J\) with Mordell--Weil rank 0, and torsion subgroup of order bounded by \(72\), and a two-torsion subgroup isomorphic to \(\mathbb{Z}/2\mathbb{Z}\).

Let \(P=[0 : -1 : 1], Q=[1 : -1 : 1]\). The divisor class [P-Q] has order 36 in the Jacobian. This means that if the torsion subgroup has order 72 then it is isomorphic to either \(\mathbb{Z}/72\mathbb{Z}\) or \(\mathbb{Z}/36\mathbb{Z}\times\mathbb{Z}/2\mathbb{Z}\). However, the second option is ruled out by the structure of the two-torsion subgroup. Therefore the only two options for the torsion subgroup are either \(\mathbb{Z}/72\mathbb{Z}\) or \(\mathbb{Z}/36\mathbb{Z}\).

The image of \([P-Q]\) under the map \(J(\mathbb{Q})\to J(\mathbb{F}_5) \cong \mathbb{Z}/72\mathbb{Z}\times\mathbb{Z}/2\mathbb{Z}\) is \emph{not} divisible by \(2\). In fact, under the identification of \(J(\mathbb{F}_5)\) with \(\mathbb{Z}/72\mathbb{Z}\times\mathbb{Z}/2\mathbb{Z}\), \([P-Q]\) is mapped to the element \((38,1)\), which is not divisible by \(2\). 

Therefore the Mordell--Weil group is isomorphic to \(\mathbb{Z}/36\mathbb{Z}\), and is generated by \([P-Q]\). The curve \(H\) has good reduction at \(5\) so that all the maps in the following commutative diagram are injective. 
\begin{center}
\begin{tikzcd}
H(\mathbb{Q}) \arrow{d}{} \arrow{r}{}
&J(\mathbb{Q}) \arrow{d}{\imath}\\
H(\mathbb{F}_5) \arrow{r}{} & J(\mathbb{F}_5)
\end{tikzcd}
\end{center}

There are six points in \(H(\mathbb{F}_5)\) (given in weighted homogeneous coordinates):
\begin{equation}
[1 : 0 : 0], [0 : 1 : 1], [0 : 4 : 1], [1 : 1 : 1], [1 : 4 : 1], [2 : 0 : 1].
\end{equation}
Of these points, only the five points \([1 : 0 : 0], [0 : 1 : 1], [0 : 4 : 1], [1 : 1 : 1], [1 : 4 : 1]\) map to the image of the map \(\imath\) in the diagram, and these exactly correspond to the five points we know on \(H\). Since the map \(H(\mathbb{Q})\to{H(\mathbb{F}_5)}\) is injective, this proves that these are the only rational points on \(H\). 

Combining the preimages of \(\psi\) together with its locus of indeterminacy, we find all rational points on \(\bar{C}_7\), and these are exactly the four points we already discovered. None of these four points correspond to \(\phi_b\) admitting the graph N3H1.

\subsubsection{N3H2}
\label{sec:orgheadline21}

For the proof that the graph N3H2 is non-admissible (up to standard conjectures) see the appendix by M. Stoll.

\subsubsection{N3H3}
\label{sec:orgheadline22}

We start with R3P5. Recall that graph R3P4 had the parametrization
\begin{equation}
\phi_d(z)=\frac{(2d - 1)z^2 - (d^2 + d - 1)z + d^2 - d}{(2d - 1)z^2}, \quad d\neq{0,1,\frac{1}{2}},
\end{equation}
where \(d\) is a periodic point of (formal) period \(2\), and the graph R3P5 was represented by the following additional condition.

\begin{equation}
(4d^2 - 4d + 1)w^2 - (d^3 - 2d + 1)w + d^3-2d^2+d = 0.
\label{eq:N3H3a}
\end{equation}

The point \(w\) was such that \(\phi^2(w) = d\) and \(\phi(w)\) is non-periodic. The second periodic point of period \(2\) is \(\frac{d-1}{2d-1}\) and its non-periodic preimage is \(\frac{d-1}{d}\). Let \(u\) be a rational preimage of \(\frac{d-1}{d}\), i.e. \(\phi(u)= \frac{d-1}{d}\). This implies

\begin{equation}
(2d - 1)u^2 - (d^3 + d^2 - d)u +d^3 - d^2 = 0.
\label{eq:N3H3b}
\end{equation}

Together with Eq. \eqref{eq:N3H3a}, we get the following affine space curve parametrizing the graph N3H3:

\begin{equation}
C_1 : 
\begin{cases}
(4d^2 - 4d + 1)w^2 - (d^3 - 2d + 1)w + d^3-2d^2+d = 0 \\
(2d - 1)u^2 - (d^3 + d^2 - d)u +d^3 - d^2 = 0.
\end{cases}
\end{equation}

Using the discriminant (with respect to \(u\)) we can bring the second equation to the form

\begin{equation}
v^2 = d^4 + 2d^3 - 9d^2 + 10d - 3
\end{equation}
where 
\begin{equation}
v = \frac{(4d-2)u-(d^3+d^2-d)}{d},
\end{equation}
and we have already seen in Section \ref{sec-R3P5} that Eq. \eqref{eq:N3H3b} can be transformed in a similar way to
\begin{equation}
y^2 = d^4 - 14d^3 + 15d^2 - 6d + 1.
\end{equation}

Thus the curve \(C_1\) is birational to a curve \(C_2\) in \(\mathbb{A}^3_{(d,y,v)}\) defined by

\begin{equation}
C_2:
\begin{cases}
y^2 = d^4 - 14d^3 + 15d^2 - 6d + 1 \\
v^2 = d^4 + 2d^3 - 9d^2 + 10d - 3
\end{cases}.
\end{equation}
The curve \(C_2\) is of genus \(5\), and is a double-cover of the following curve:

\begin{equation}
H_1 : s^2 = (d^4 - 14d^3 + 15d^2 - 6d + 1)(d^4 + 2d^3 - 9d^2 + 10d - 3).
\end{equation}

The curve \(H_1\) has genus \(3\), and its equation can be simplified to the following model:

\begin{equation}
H_2 : y^2 = -3x^8 + 4x^7 - 2x^6 - 16x^5 + 11x^4 + 16x^3 - 2x^2 - 4x - 3.
\end{equation}

This curve has automorphism group \(\mathbb{Z}/2\mathbb{Z}\times\mathbb{Z}/2\mathbb{Z}\). When quotienting the curve by the automorphism \([Z : -Y : -X]\), we get the following curve:

\begin{equation}
H_3 : y^2 = -3x^6 + 4x^5 - 26x^4 + 12x^3 - 55x^2 - 16x + 4.
\end{equation}

This genus \(2\) hyperelliptic curve contains (at least) the following two points:

\begin{equation}
P_1 = [0 : -2 : 1], P_2 = [0 : 2 : 1].
\end{equation}

Using \textsc{Magma}, one can compute that \(H_3\) has Jacobian with Mordell--Weil rank bounded by \(1\), and torsion subgroup isomorphic to \(\mathbb{Z}/2\mathbb{Z}\). The divisor class \([P_1 - P_2]\) has infinite order in the Jacobian, and therefore generates a finite index subgroup of the Mordell--Weil group. We can use the implementation of the Chabauty--Coleman/Mordell--Weil sieving algorithm (cf. Bruin and Stoll \cite{MR2685127}) in \textsc{Magma}, and find that the only rational points on the curve are the two points we had already discovered. 

We pull these two points all the way back \(C_1\) to find the complete set of rational points on the latter curve (not including \(4\) points at infinity):
\begin{align*}
(1/2 , 1 , 1), (1 , 0 , 0), (0 , 0 , 0), (1 , 1 , 0), (0 , 0 , 1).
\end{align*}
None of these points have a \(d\) value corresponding to a quadratic rational function \(\phi_d\) realizing N3H3.

\subsection{Support for Conjecture 1 \label{support-conj1}}
\label{sec:orgheadline26}

\subsubsection{Extra 3-cycle}
\label{sec:orgheadline24}
We compute the 3-rd dynatomic polynomial of \(\phi_a\). 
\begin{align*}
\Phi_{a,3}^*(z) := \Phi_{\phi_a,3}^*(z) = &(a^5 + 5a^4 + 10a^3 + 10a^2 + 5a + 1)z^5 \\
& + (-2a^5 - 7a^4 - 9a^3 - 4a^2 + a + 1)z^4 \\
& + (a^5 - 7a^3 - 13a^2 - 10a - 3)z^3  \\
& + (2a^4 + 5a^3 + 4a^2 + a)z^2  \\
& + (a^3 + 3a^2 + 3a + 1)z
\end{align*}
We divide the dynatomic polynomial by the factor \((a+1)z(z-1)\) to get the irreducible polynomial
\[P_3 := a^3z^3 - a^3z^2 + 3a^2z^3 - 2a^2z + 3az^3 + 2az^2 - 2az - a + z^3 + 2z^2- z - 1\]

The curve \(C_3\) defined by the equation \(P_3 = 0\) contains the four points

\[[0 : 1 : 0], [-1 : 0 : 1], [1 : 0 : 0], [-1 : 1 : 1].\]

Furthermore, \(C_3\) is birational to the genus 2 hyperelliptic curve defined by
\[H_3: y^2 = x^6 - 4x^5 + 6x^4 - 2x^3 + x^2 - 2x + 1\]

Using \textsc{Magma}, we can check that \(H_3\) has Jacobian of Mordell--Weil rank 0. The curve \(H_3\) has good reduction modulo 3, and its reduction has Jacobian isomoprphic to the group \(\mathbb{Z}/19\mathbb{Z}\). This means that Mordell--Weil group (of the Jacobian) of \(H_3\) is either trivial or isomorphic to \(\mathbb{Z}/19\mathbb{Z}\). However, one can check that the element on the Jacobian defined by the difference of the two points at infinity \([1:-1:0] - [1:1:0]\) has order \(19\) and therefore the Jacobian of \(H_3\) is isomorphic to \(\mathbb{Z}/19\mathbb{Z}\).

We get the following commutative diagram where all the maps are injective.

\begin{center}
\begin{tikzcd}
H_3(\mathbb{Q}) \arrow{d} \arrow{r}
&J(\mathbb{Q}) \arrow{d}\\
H_3(\mathbb{F}_3) \arrow{r} & J(\mathbb{F}_3)
\end{tikzcd}
\end{center}

We check that 
\[H_3(\mathbb{F}_3) = \{[1 : -1 : 0], [1 : 1 : 0], [0 : -1 : 1], [0 : 1 : 1], [1 : -1 : 1], [1 : 1 : 1]\}.\]
All of these points are reductions of points from \(H_3(\mathbb{Q})\), and this implies that the \(\mathbb{Q}\)-rational points on \(H_3\) are exactly the six points
\[ [1 : -1 : 0], [1 : 1 : 0], [0 : -1 : 1], [0 : 1 : 1], [1 : -1 : 1], [1 : 1 : 1].\]

We pull these points back to \(C_3\), to find that the only rational points on \(C_3\) are the four points we already discovered. None of which correspond to non-degenerate maps \(\phi_a\).

\subsubsection{Extra 4-cycle \label{sec-4cycle}}
\label{sec:orgheadline25}
We compute the 4th dynatomic polynomial of \(\phi_a\), and divide by a factor of \((a+1)^4\), to get the following polynomial.

\begin{align*}
P_4(a,t) := &(a^{6} + 6 a^{5} + 15 a^{4} + 20 a^{3} + 15 a^{2} + 6 a + 1) z^{12} \\
& + (-2 a^{7} - 19 a^{6} - 64 a^{5} - 107 a^{4} - 98 a^{3} - 49 a^{2} - 12 a - 1) z^{11}  \\
& + (a^{8} + 16 a^{7} + 77 a^{6} + 152 a^{5} + 128 a^{4} + 16 a^{3} - 46 a^{2} - 30 a - 6) z^{10}  \\
& + (-3 a^{8} - 27 a^{7} - 58 a^{6} + 38 a^{5} + 248 a^{4} + 304 a^{3} + 165 a^{2} + 41 a + 4) z^{9}  \\
& + (3 a^{8} + 8 a^{7} - 61 a^{6} - 239 a^{5} - 240 a^{4} - 8 a^{3} + 111 a^{2} + 59 a + 9) z^{8}  \\ 
& + (-a^{8} + 11 a^{7} + 64 a^{6} + 10 a^{5} - 251 a^{4} - 317 a^{3} - 126 a^{2} - 8 a + 2) z^{7}  \\
& + (-6 a^{7} + 12 a^{6} + 128 a^{5} + 133 a^{4} - 94 a^{3} - 153 a^{2} - 48 a - 1) z^{6}  \\
& + (-16 a^{6} - 5 a^{5} + 130 a^{4} + 156 a^{3} - a^{2} - 42 a - 8) z^{5}  \\
& + (-25 a^{5} - 26 a^{4} + 78 a^{3} + 94 a^{2} + 12 a - 7) z^{4}  \\ 
& + (-25 a^{4} - 31 a^{3} + 26 a^{2} + 33 a + 5) z^{3}  \\
& + (-16 a^{3} - 20 a^{2} + 2 a + 5) z^{2}  \\
& + (-6 a^{2} - 7 a - 1) z - a - 1.
\end{align*}

The equation \(P_4 = 0\) describes an affine plane curve of genus 12 which we denote by \(C_4\). We follow the trace map method described in \cite{MR1665198}: Let 
\[tr_4(z)=z+\phi(z)+\phi^2(z)+\phi^3(z).\] 
The image of the curve \(C_4\) under the map \(\Psi: (a,z) \mapsto (a, tr_4(z))\) is birational to the quotient curve \(C_4/\left<\sigma\right>\), where \(\sigma\) is the automorphism of \(C_4\) defined by \((a,z)\mapsto(a,\phi(z))\). When \(\phi\) is a polynomial, the image of \(\Psi\) can be calculated by taking the resultant of \(\Phi_{a,4}^*(z)\) and \(t-tr_4(z)\) with respect to \(z\). However, in our case \(t-tr_4(z)\) is not a polynomial, since \(tr_4(z)\) is a rational function in \(z\). We fix this by clearing out the denominator of \(tr_4(z)\); we denote the numerator of \(tr_4\) by \(A\) and the denominator by \(B\). 

We compute the resultant \(\text{Res}(P_4, Bt-A)\) and find the following irreducible factor.
\begin{multline*}
\tilde{P}_4(a,t) = a^{7} t - 2 \, a^{6} t^{2} + a^{5} t^{3} - a^{7} + 13 \, a^{6} t - 17 \, a^{5} t^{2} + 5 \, a^{4} t^{3} - \\
7 \, a^{6} + 53 \, a^{5} t - 47 \, a^{4} t^{2} + 10 \, a^{3} t^{3} - 6 \, a^{5} + 80 \, a^{4} t - \\
60 \, a^{3} t^{2} + 10 \, a^{2} t^{3} + 43 \, a^{4} + 42 \, a^{3} t - 38 \, a^{2} t^{2} + 5 \, a t^{3} +\\ 
95 \, a^{3} - 10 \, a^{2} t - 11 \, a t^{2} + t^{3} + 89 \, a^{2} - 19 \, a t - t^{2} + 42 \, a - 6 \, t + 9
\end{multline*}

The degree \(8\) curve defined by \(\tilde{P}_4=0\) is of genus 1. It is birational to the elliptic curve with Cremona reference 19a3 (this curve has already been encountered as being birational to the modular curve of N3E1, see the remark in Section \ref{sec-N3E1}), and has the following minimal model
\[y^2 + y = x^3 + x^2 + x.\]

Its Mordell--Weil group is isomorphic to \(\mathbb{Z}/3\mathbb{Z}\) and therefore there are exactly three rational points on this elliptic curve. These are
\[[0 : 0 : 1], [0 : 1 : 0], [0 : -1 : 1].\]

Pulling these points back to the curve \(C_4/\left<\sigma\right>\), we find all rational points on \(C_4/\left<\sigma\right>\) to be
\[[0 : 1 : 0], [1 : 0 : 0], [1 : 1 : 0].\]
All these points are "cusps" (the points added to the dynamical modular curve to make it projective), and therefore we can conclude that there are no quadratic rational functions \(\phi_a\) with a 4-cycle.

\bibliographystyle{plain}

\appendix

\title{Appendix. Rational points on a curve of genus 6}
\author{Michael Stoll}


\maketitle

\begin{abstract}
  We determine the set of rational points on the curve of
  genus~6 that parameterizes quadratic maps $\phi \colon \PP^1 \to \PP^1$
  with an orbit of length~3 containing a critical point
  and a point~$P$ such that $\phi^{\circ 3}(P)$ has order~2
  (but $\phi^{\circ 2}(P)$ is not periodic).
  The result is conditional on standard conjectures (including BSD)
  on the $L$-series of the Jacobian of the curve in question.
\end{abstract}


\section{Introduction}

We consider the curve~$C$ that classifies (up to conjugation by an automorphism
of~$\PP^1$) quadratic maps $\phi \colon \PP^1 \to \PP^1$ such that
\begin{enumerate}[1.]
  \item $\phi$ has a cycle of length~3 containing a critical point, and
  \item $\phi$ has a cycle of length~2, together with a marked third
        preimage of one of the two points in the cycle (whose second
        image is not in the cycle).
\end{enumerate}
(See the introduction of the main article for the definitions of a critical point and \(n\)-cycle). These are exactly the type of maps described by the graph N3H2 in Table \(4\) of the main article. The goal is to show that no such (non-degenerate) $\phi$ are defined
over~$\Q$, which is equivalent to showing that all the rational points
on~$C$ correspond to degenerate maps~$\phi$.
To this end, we first derive an equation for~$C$ as a singular affine
plane curve, then we construct its canonical model in~$\PP^5$,
which is a smooth projective curve~$D$ birational to~$C$.
Finally, we determine~$D(\Q)$ (to be able to do this, we have to
assume some standard conjectures regarding the $L$-series of the
Jacobian of~$D$, including the BSD conjecture)
and map the points back to~$C$.

The arguments and computations are to a large extent parallel to those
performed in~\cite{Stoll2008} in a similar situation, so we give here
a condensed description and refer the reader to~\cite{Stoll2008} for
more information.


\section{The curve and its canonical model}

We can fix the critical \(3\)-cycle to be $0 \mapsto \infty \mapsto 1 \mapsto 0$
with $0$ the critical point (see Section 4.1.2 in the main article for details). Then
\[ \phi(z) = \frac{(a+1) z^2 - a z - 1}{(a+1) z^2} \]
with a parameter~$a$. The condition for the existence of a rational
2-cycle is then that $(a+1)^2 + 4$ is a square. The conic given by
this condition can be parameterized, leading to
\[ a = -\frac{4t}{t^2-1} - 1 \]
and
\[ \phi_t(z) = \frac{4 t z^2 - (t^2 + 4 t - 1) z + t^2 - 1}{4 t z^2}; \]
the 2-cycle contains the two points $\frac{t+1}{2}$ and~$\frac{t-1}{2t}$,
which are swapped by the automorphism (on the parameter space) $t \mapsto -1/t$.
$\phi_t$ degenerates for $t = -1, 0, 1, \infty$.

The other preimage of~$\frac{t+1}{2}$ is~$-\frac{t+1}{t-1}$.
We will require it to have a rational second preimage.
The requirement of a rational (first) preimage results in a curve
that is birational to the elliptic curve~$E$ with label~53a1 in the
Cremona database, whose Mordell-Weil group is~$\Z$ (This is the curve parametrizing rational functions admitting the graph R3P5 in Table 1 of the main article; see Section 4.1.6 there).
Requiring a second rational preimage results in a double cover of this
curve, which (from setting $\phi_t^{\circ 2}(z) = -\frac{t+1}{t-1}$)
can be given by the affine equation
\begin{align*}
  16 t^2 (t+1) z^4 &+ (t^2 + 4 t - 1) (t^3 - 13 t^2 - 5 t + 1) z^3
    + (t^5 + 29 t^4 + 34 t^3 - 30 t^2 - 3 t + 1) z^2 \\
    &{}- 4 (t - 1) t (t + 1) (t^2 + 4 t - 1) z
    + 2 (t - 1)^2 t (t + 1)^2 = 0 .
\end{align*}
A quick computation in Magma~\cite{Magma} shows that (the smooth projective
model of) this curve has genus~$g = 6$. Setting
\[ t = \frac{u+1}{u-1} \qquad\text{and}\qquad z = \frac{1}{1-v} \]
results in the simpler equation
\begin{align*}
  F(u,v) &:= u^5 v^2 + 2 u^4 v^3 - u^4 v^2 - u^4 v + u^3 v^4 - 4 u^3 v^2 + u^3 \\
         &\qquad{} + u^2 v^4 - 4 u^2 v^3 + 3 u^2 v - 2 u v^3 + 4 u v^2 - u + v^2 - v = 0
\end{align*}
for our curve~$C$.
In terms of $u$ and~$v$, the involution corresponding to the double
cover $C \to E$ is given by $(u,v) \mapsto (u,u^{-1}-1-u-v)$.

Set
\[ \omega_0 = \frac{du}{\frac{\partial F}{\partial v}(u,v)}
            = -\frac{dv}{\frac{\partial F}{\partial u}(u,v)} .
\]
Then the following differentials are a basis of the space of regular
differentials on~$C$:
\begin{align*}
  \omega_1 &= (u^3 v + u^2 v^2 - 2 u^2 v - u v^2 + v) \omega_0, \\
  \omega_2 &= (u^2 v + u v^2 - u v - v) \omega_0, \\
  \omega_3 &= u v \omega_0, \\
  \omega_4 &= (u^2 - u) \omega_0, \\
  \omega_5 &= (u - 1) \omega_0, \\
  \omega_6 &= \omega_0.
\end{align*}
The projective closure~$D$ of the image of the corresponding canonical map $C \to \PP^5$
(with coordinates $w_1, \ldots, w_6$) is then defined by the following
six quadratic equations:
\begin{align*}
  -w_2 w_5 + w_1 w_6 &= 0, \\
  -w_2 w_4 + w_1 w_5 + w_2 w_5 &= 0, \\
  -w_5^2 + w_4 w_6 - w_5 w_6 &= 0, \\
  -w_3^2 - w_3 w_4 + w_1 w_6 + w_2 w_6 + w_3 w_6 &= 0, \\
  w_1^2 + 2 w_1 w_2 - w_1 w_4 - w_2 w_4 + w_4^2 - w_1 w_5 + w_4 w_5 &= 0, \\
  w_1 w_2 + 2 w_2^2 - w_2 w_4 - w_2 w_5 + w_4 w_5 + w_5^2 - w_1 w_6
    - w_2 w_6 + w_4 w_6 + w_5 w_6 &= 0.
\end{align*}
The birational map $D \to C$ is given by
\[ u = \frac{w_4+w_5+w_6}{w_5+w_6}, \qquad v = \frac{w_3}{w_5+w_6} . \]

Using the \verb+PointSearch+ command of Magma, we find the following nine
rational points on~$D$:
\begin{align*}
  P_1 &= (0 : 0 : 0 : 2 : -2 : 1), \\
  P_2 &= (0 : 0 : 1 : 0 : 0 : 1), \\
  P_3 &= (0 : 0 : -1 : 2 : -2 : 1), \\
  P_4 &= (0 : 0 : 0 : 0 : -1 : 1), \\
  P_5 &= (1 : -1 : 0 : 0 : -1 : 1), \\
  P_6 &= (-2 : 1 : 0 : 0 : 0 : 0), \\
  P_7 &= (0 : 0 : 0 : 0 : 0 : 1), \\
  P_8 &= (0 : 0 : 1 : 0 : -1 : 1), \\
  P_9 &= (1 : -1 : 1 : 0 : -1 : 1). \\
\end{align*}
We suspect that these are all the rational points. The proof of this claim
will take up the remainder of this note.

We note that the images of these points on~$C$ are $(-1,0)$, $(1,1)$, $(-1,1)$,
$(0,0)$, $(0,1)$, a point at infinity, $(1,0)$, and two times a point at
infinity. Since $u = -1,0,1,\infty$ gives a degenerate~$\phi$ and $v = \infty$
gives $z = 0$, which also makes $\phi$ degenerate, this will imply that there
are no non-degenerate~$\phi$ defined over~$\Q$ with the required properties.


\section{The Jacobian}

Working with the equations for~$D \subset \PP^5$, we find that $D$ has
bad reduction (at most) at the primes $3$, $53$ and~$99\,563$.
In each of these cases, the reduction is semistable, with a single component
in the special fiber that has one split node for $p = 3$, two non-split nodes
each defined over~$\F_{53}$ for $p = 53$, and one split node  for $p = 99\,563$.
(A node is \emph{split}, if the two tangent directions are defined over
the field of definition of the node.) So $D$ is semistable, and its Jacobian~$J$
has conductor $N_J = 3 \cdot 53^2 \cdot 99\,563$. Since $D$ maps non-trivially to
the elliptic curve~$E$ of conductor~$53$, $J$ is isogenous to a product
$E \times A$, where $A$ is an abelian variety over~$\Q$ of dimension~$5$
and with conductor $N_A = 3 \cdot 53 \cdot 99\,563 = 15\,830\,517$.

With a computation analogous to that leading to Lemma~4 in~\cite{Stoll2008},
we find that the torsion subgroup of~$J(\Q)$ has exponent dividing~$2$
(we did not try to find the torsion subgroup exactly) and that the differences
of the nine known rational points on~$D$ generate a subgroup isomorphic
to~$\Z^2$ of~$J(\Q)$. So if we can show that $J(\Q)$ has rank~$2$, then
we know generators of a subgroup of finite index (and the rank is strictly
less than the genus), so we can apply the Chabauty-Coleman method.

Since there is little hope to perform a successful Selmer group computation
on~$J$, which would give an upper bound for the rank, we follow the approach
already used in~\cite{Stoll2008} and assume that the $L$-series of~$J$
has an analytic continuation to all of~$\C$, that the function
\[ \Lambda(J, s) = N_J^{s/2} (2\pi)^{-6s} \Gamma(s)^6 L(J, s) \]
satisfies the functional equation $\Lambda(J, 2-s) = w_J \Lambda(J, s)$
with the global root number $w_J = \pm 1$, and that the Birch and Swinnerton-Dyer
conjecture holds for~$J$. For the computations, it is better to work with
the $L$-series of~$A$, since its conductor $N_A$ is smaller than~$N_J$.
According to Magma's implementation of $L$-series, it requires the
coefficients of~$L(A,s)$ up to $n \approx 105\,000$ for a precision of
20~decimal digits. We find the coefficients of the Euler factors
of the $L$-series of~$J$ up to the required bound by counting the
$\F_q$-points on~$D$ for all prime powers~$q$ below the bound.
This is most efficiently done on the affine model~$C$, by keeping track
of the points at infinity modulo~$p$ and of what happens at the
singular points (and some care has to be taken at the bad primes).
In this way and after dividing by the Euler factors of~$E$, we obtain
within a few hours the relevant coefficients for the computation.
We then check (using Magma's \verb+CheckFunctionalEquation+) that the
data we have computed is compatible with the expected functional equation
for $L(A,s)$ with root number $w_A = -1$ (but not with $w_A = 1$).
So we can safely assume that $L(J,s)$ satisfies the functional equation
with root number $w_J = w_A w_E = (-1)(-1) = 1$. This is also in agreement
with the expectation that the global root number should be equal to the
product of the local root numbers, which in the semistable case is
$(-1)^{g+s}$, where $s$ is the total number of Frobenius orbits of
split nodes. Here $g = 6$ and $s = 2$, so the root number should indeed
by~$1$. We then evaluate the derivative of~$L(A,s)$ at~$s = 1$ numerically
and find a clearly nonzero value of $\approx 0.026803015530623712948$.
So according to the BSD conjecture, the rank of~$A(\Q)$ should be~$1$
and the rank of~$J(\Q)$ should therefore be~$2$.


\section{Applying the Chabauty-Coleman method}

By the results of the previous section (assuming BSD for~$J$), we now know
that the differences of the known rational points provide us with generators
of a finite-index subgroup of~$J(\Q)$. To apply Chabauty-Coleman, we have
to fix a prime~$p$. We choose $p = 5$, because $5$ is a prime of good reduction
and the set of known rational points maps bijectively to the points in~$D(\F_5)$
under reduction. This latter observation implies that it will be enough to show that
each residue class mod~$5$ in~$D(\Q_5)$ contains at most one rational point.
By the results of~\cite{Stoll2006}, this follows when we can show that
for each point in~$D(\F_5)$, there is a differential in the annihilator
of~$J(\Q)$ whose reduction mod~$5$ does not vanish there.

We first have to find a basis of this annihilator in the space of regular
differentials on~$D$ over~$\Q_5$. We first fix one of our known rational
points~$P$ as a base-point; we have to choose it in such a way that its reduction~$\bar{P}$
mod~$5$ is non-special in the sense that the Riemann-Roch space of~$6\bar{P}$
is one-dimensional. We check that $P_4$ satisfies this requirement.
We then use the reduction map $J(\Q) \to J(\F_5)$ to find two independent
points in its kernel, which we represent by divisors of the form
$\calD_j - 6 P_4$, where $\calD_1$ and~$\calD_2$ are effective of degree~$6$.
Since $\bar{P}_4$ is non-special, it follows that $\calD_j$ reduces mod~$5$
to~$6 \bar{P}_4$ for $j = 1, 2$. We choose $t = 1 + w_5/w_6$ as a uniformizer
at~$P_4$ (that reduces to a uniformizer at~$\bar{P}_4$), express the
differentials $\omega_1, \ldots, \omega_6$ as power series in~$t$ times~$dt$,
integrate formally, and use the method explained in~\cite{Stoll2008} to
compute the relevant integrals modulo~$5^3$ (modulo~$5^2$ would actually be
sufficient). They are all multiples of~$5$, so we divide them by~$5$
and reduce the $2 \times 6$-matrix obtained modulo~$5$. Its kernel
gives the reduction of annihilator, which in our case is generated by
the reductions of $\omega_1 + \omega_4$, $\omega_2 - \omega_4$, $\omega_5$
and~$\omega_6$. Since the locus of vanishing of a linear combination
of the~$\omega_j$ on~$D$ is given by the corresponding hyperplane section,
all we have to do is to check that the line defined in~$\PP^5_{\F_5}$
by $w_1 + w_4 = w_2 - w_4 = w_5 = w_6 = 0$ does not meet $D_{\F_5}$,
which is easily verified. This concludes the proof.


\renewcommand{\refname}{References to the Appendix}

\end{document}